\documentclass[10pt,a4paper]{article}

\usepackage[a4paper,left=2.5cm,right=2.5cm,
top=2.5cm,bottom=2.5cm]{geometry}

\usepackage{amsmath,amssymb,amsfonts,amsthm}
\usepackage{mathrsfs}
\usepackage{graphicx}
\usepackage{booktabs}
\usepackage{multirow}
\usepackage{xcolor}
\usepackage{enumerate}
\usepackage{appendix}
\usepackage{hyperref}

\newtheorem{theorem}{Theorem}[section]

\newtheorem{lemma}[theorem]{Lemma}
\newtheorem{corollary}[theorem]{Corollary}

\theoremstyle{definition}
\newtheorem{definition}[theorem]{Definition}
\newtheorem{example}[theorem]{Example}

\theoremstyle{remark}

\numberwithin{equation}{section}

\title{\bfseries
Casorati inequalities for Riemannian submersion along mixed
distributions and their applications
}

\author{Ravindra Singh$^{1}$ and Mukut Mani Tripathi$^{2^{\ast}}$ }

\date{
\small
$^{1}$Department of Mathematics, Banaras Hindu University, Varanasi, India\\
\textit{Email:} \href{mailto:khandelrs@bhu.ac.in}{khandelrs@bhu.ac.in} \\
ORCID: 0009-0009-1270-3831 \\
$^{2}$Department of Mathematics, Banaras Hindu University, Varanasi, India\\
\textit{Email:} \href{mailto:mmtripathi66@yahoo.com}{mmtripathi66@yahoo.com} \\
ORCID: 0000-0002-6113-039X
}

\begin{document}
\maketitle

\abstract{\noindent This paper introduces Casorati inequalities for the normalised scalar curvature and normalised Casorati curvature of vertical and horizontal distributions for Riemannian submersions between Riemannian manifolds. We completely characterise the equality cases from both algebraic and geometric perspectives. As applications, we derive the corresponding inequalities for Riemannian submersions from real, complex, and generalised Sasakian space forms, including Sasakian, cosymplectic, Kenmotsu, and almost $C(\alpha)$ space forms. We provide several examples to demonstrate the effectiveness and applicability of the results obtained.}

\medskip




\maketitle


\section{Introduction}

In 1889, Felice Casorati introduced a curvature now widely recognized as the Casorati curvature.
This curvature applies to regular surfaces in Euclidean $3$-space and is characterised as the normalised sum of the squared principal curvatures \cite{Casorati_1890_AM}. Casorati favoured this curvature over the traditional Gaussian curvature because the Casorati curvature vanishes for a surface in
Euclidean $3$-space if and only if both Euler normal curvatures (also known as
principal curvatures) of the surface vanish simultaneously. Thus, it
aligns more closely with the common intuition of curvature. For a
hypersurface in a Riemannian manifold, the Casorati curvature is defined as
the normalised sum of the squared principal normal curvatures of the
hypersurface. More generally, the Casorati curvature of a submanifold of a
A Riemannian manifold is defined as the normalised second fundamental form \cite{Decu_Haesen_Verstraelen_08_JIPAM}. Geometrical
meaning and significance of the Casorati curvature, as discussed by several
geometers, are presented in survey articles \cite{Chen_Survey} and the
references cited therein.

In addition to submanifold theory, the investigation of geometric maps that
preserve metric structures has emerged as a significant area of research
within differential geometry. During the 1960s, O'Neill \cite{Neill_66_MMJ}
and Gray \cite{Gray_1967} independently introduced the concept of
Riemannian submersions, which provide a natural framework for relating the
geometric properties of two Riemannian manifolds. A surjective smooth map 
\[
\pi:(M_{1},g_{1})\rightarrow (M_{2},g_{2}), 
\]
between Riemannian manifolds is called a Riemannian submersion if $%
\pi_{\ast} $ is onto and satisfies 
\[
g_{2}(\pi_{\ast}X_{1},\pi_{\ast}X_{2}) = g_{1}(X_{1},X_{2}) 
\]
for all horizontal vector fields $X_{1},X_{2}\in {\cal H}$, where $%
\pi_{\ast} $ denotes the differential map of $\pi$.

Riemannian submersions play a significant role in both mathematics and
physics, particularly in Yang-Mills theory \cite{BL_1981}, Kaluza-Klein
theory \cite{Bourguignon_1989,IV_1987}, as well as in supergravity and superstring
theory \cite{IV_1991,Mustafa_2000}. Motivated by Casorati
inequalities for submanifolds, several researchers have established
analogous inequalities for Riemannian submersions. For instance, Casorati inequalities for Riemannian
submersions were first established in \cite{Lee_Lee_Sahin_Vilcu_21_AMPA} for
real and complex space forms equipped with vertical and horizontal
distributions. These results were subsequently extended to Sasakian space forms in
\cite{Polat_Lee_Sahin_25_JGP}. More recently, Singh {\em et al.} \cite%
{Singh_Meena_Meena_26_JMAA} established general Casorati inequalities
associated with the vertical and horizontal distributions of Riemannian
submersions and investigated their applications to several classes of space
Building upon these developments, the present paper establishes a
Casorati inequality for Riemannian submersions associated with mixed
distributions and characterizes the cases of equality from both geometric and algebraic perspectives.
Additionally, corresponding inequalities are derived for
Riemannian submersions whose total spaces are real, complex, and generalized
Sasakian space forms. In particular, applications are presented for Sasakian,
cosymplectic, Kenmotsu, and C($\alpha$)-space forms. Several examples are provided to illustrate the results and their applications.

The structure of this research paper is as follows. Section \ref{section 2} presents the necessary definitions and fundamental results on Riemannian submersions and Riemannian curvature tensor fields. Section \ref{section 3} establishes Casorati inequalities for Riemannian submersions with mixed distributions and discusses their equality cases through appropriate examples. Finally, Section~\ref{sec_4} addresses applications of the derived inequalities to Riemannian submersions whose total manifolds are real, complex, and generalized Sasakian space forms, with the equality cases illustrated by several examples. Direct computations yield the corresponding inequalities for Riemannian submersions whose total spaces are Sasakian, Kenmotsu, cosymplectic, and $C(\alpha)$-space forms.

\section{Preliminaries \label{section 2}}

Let $\left( M_{1},g_{1}\right) $ and $\left( M_{2},g_{2}\right) $ be
Riemannian manifolds of dimensions $n_{1}$ and $n_{2}$, respectively. A
smooth surjective map 
\[
\pi :\left( M_{1},g_{1}\right) \longrightarrow \left( M_{2},g_{2}\right) 
\]%
is called a {\it Riemannian submersion} if the differential amp $\pi _{\ast
p}:T_{p}M_{1}\rightarrow T_{\pi (p)}M_{2}$ is surjective for every $p\in
M_{1}$ and satisfies 
\[
g_{2}\left( \pi _{\ast p}X,\pi _{\ast p}Y\right) =g_{1}(X,Y) 
\]%
for all horizontal vectors $X,Y\in {\cal H}_{p}$. The kernel of $\pi _{\ast
} $ defines the {\it vertical distribution }${\cal V}=\ker \pi _{\ast }$ and
its orthogonal complement ${\cal H}=(\ker \pi _{\ast })^{\perp }$ is called
the {\it horizontal distribution}. Thus, we have the orthogonal
decomposition
\[
TM_{1}={\cal V}\oplus {\cal H}. 
\]%
For each $p\in M_{1}$, the corresponding vertical and horizontal spaces are
given by 
\[
{\cal V}_{p}=\ker \pi _{\ast p},\quad {\cal H}_{p}=(\ker \pi _{\ast
p})^{\perp }. 
\]%
Let $\{V_{1},\ldots ,V_{\ell },U_{1},\ldots ,U_{s}\}$ be an orthonormal
basis of $T_{p}M_{1}$ such that $\{V_{1},\ldots ,V_{\ell }\}$ and $%
\{U_{1},\ldots ,U_{s}\}$ are orthonormal bases of ${\cal V}_{p}$ and ${\cal H%
}_{p}$, respectively. We shall use these bases throughout this paper.

\subsubsection{O'Neill tensors}

The geometry of Riemannian submersions is characterised by O'Neill's tensors 
{${\cal T}$} and {${\cal A}$} defined by 
\begin{eqnarray*}
{\cal T}\left( E,F\right) &=&{{\cal T}_{E}F=h\nabla _{vE}vF+v\nabla _{vE}hF,}
\\
{\cal A}\left( E,F\right) &=&{\cal A}_{E}F=v{\nabla _{hE}hF+h\nabla _{hE}vF,}
\end{eqnarray*}%
where $E,F\in TN_{1}$ and $\nabla $ is the Levi-Civita connection on $g_{1}$%
, while $h$ and $v$ are projection morphisms of $E,F\in TN_{1}$ to ${\cal H}$
and ${\cal V}$, respectively. We also have%
\begin{eqnarray*}
{{\cal T}^{{\cal H}}} &:&{\cal V}\times {{\cal V}\rightarrow {\cal H}}, \\
{{\cal T}^{{\cal V}}} &{:}&{{\cal V}\times {{\cal H}\rightarrow {\cal V}}},
\\
{{\cal A}^{{\cal H}}} &{:}&{{\cal H}\times {{\cal V}\rightarrow {\cal H}}},
\\
{{\cal A}^{{\cal V}}} &{:}&{{\cal H}\times {{\cal H}\rightarrow {\cal V}}}.
\end{eqnarray*}%
The O'Neill tensors also satisfy the folowing Properties: 
\[
{\cal A}^{\cal V}_{Y_{1}}Y_{2}=-{\cal A}^{\cal V}_{Y_{2}}Y_{1},\quad {\cal T}^{\cal H}_{U_{1}}U_{2}=%
{\cal T}^{\cal H}_{U_{2}}U_{1}, 
\]%
\[
g_{1}\left( {\cal T}_{E}F,G\right) =-g_{1}\left( F,{\cal T}_{E}G\right)
,\quad g_{1}\left( {\cal A}_{E}F,G\right) =-g_{1}\left( F,{\cal A}%
_{E}G\right) , 
\]%
where $U_{1},U_{2}\in {\cal V}$, $Y_{1},Y_{2},Y_{3}\in {\cal H}$ and $E$,
\thinspace $F$, $G\in TM_{1}$ \cite{Neill_66_MMJ}.

\subsubsection{Relations between Riemann curvature tensor fields}

Let $R^{M_{1}}$, $R^{M_{2}}$, $R^{{\cal V}}$, and $R^{{\cal H}}$ denote the
Riemann curvature tensor fields corresponding to $M_{1}$, $M_{2}$, ${\cal %
V}$ and ${\cal H}$, respectively. Then, 
\begin{eqnarray}
R^{{M_{1}}}\left( F_{1},F_{2},F_{3},F_{4}\right) &=&R^{{\cal V}}\left(
F_{1},F_{2},F_{3},F_{4}\right) -g_{1}(T_{F_{1}}F_{4},T_{F_{2}}F_{3}) 
\nonumber \\
&&+g_{1}(T_{F_{2}}F_{4},T_{F_{1}}F_{4}),  \label{Gauss_Codazz_RS_T}
\end{eqnarray}%
\begin{eqnarray}
R^{M_{1}}\left( X_{1},X_{2},X_{3},X_{4}\right) &=&R^{{\cal H}}\left(
X_{1},X_{2},X_{3},X_{4}\right) +2g_{1}\left(
A_{X_{1}}X_{2},A_{X_{3}}X_{4}\right)  \nonumber \\
&&-g_{1}\left( A_{X_{2}}X_{3},A_{X_{1}}X_{4}\right) +g_{1}\left(
A_{X_{1}}X_{3},A_{X_{2}}X_{4}\right) ,  \label{Gauss_Codazz_RS_A}
\end{eqnarray}%
\begin{eqnarray}
R^{M_{1}}\left( X_{1},F_{1},X_{2},F_{2}\right) &=&g_{1}\left( \left( \nabla
_{X_{1}}{\cal T}\right) \left( F_{1},F_{2}\right) ,X_{2}\right) +g_{1}\left(
\left( \nabla _{F_{1}}{\cal A}\right) \left( X_{1},X_{2}\right) ,F_{2}\right)
\nonumber \\
&&-g_{1}\left( {\cal T}_{F_{1}}X_{1},{\cal T}_{F_{2}}X_{2}\right)
+g_{1}\left( {\cal A}_{X_{2}}F_{2},{\cal A}_{X_{1}}F_{1}\right) ,
\label{eq-P-(12)}
\end{eqnarray}%
where $X_{1}$, $X_{2}$, $X_{3}$, $X_{4}\in {\cal H}$\ and $F_{1}$, $F_{2}$, $%
F_{3}$, $F_{4}\in {\cal V}$. Here, $\nabla $ is the Levi-Civita connection
with respect to the metric $g_{1}$ \cite{Falcitelli_2004,Neill_66_MMJ}.

\begin{definition}
The Kulkarni-Nomizu product $T _{1}\circledast T _{2}$ of $\left( 0,2\right) 
$-tensor fields $T _{1}$ and $T _{2}$ in a smooth manifold $M$ is a $\left(
0,4\right) $-tensor field defined by \cite{Tripathi_25_CMAMS} 
\begin{eqnarray}
\left( T _{1}\circledast T _{2}\right) \left( X,Y,Z,W\right) &=&T _{1}\left(
Y,Z\right) T _{2}\left( X,W\right) -T _{1}\left( X,Z\right) T _{2}\left(
Y,W\right)  \nonumber \\
&&+T _{2}\left( Y,Z\right) T _{1}\left( X,W\right) -T _{2}\left( X,Z\right)
T _{1}\left( Y,W\right)  \label{eq-KN-product}
\end{eqnarray}%
for all vector fields $X$, $Y$, $Z$, $W$ on $M$.
\end{definition}

\begin{definition}
The symmetric product of any two $\left( 0,2\right) $-tensor fields $T _{1}$
and $T _{2}$ is a $\left( 0,4\right) $-tensor field $T _{1}\circledcirc T
_{2}$ defined by \cite{Tripathi_25_CMAMS} 
\begin{equation}
T _{1}\circledcirc T _{2}=T _{1}\circledast T _{2}+T _{2}\circledast T _{1}.
\label{eq-KN-symm-product}
\end{equation}
\end{definition}

Now, let $\left( M,g\right) $ be an $n$-dimensional Riemannian manifold. Let 
$T$ be a Kulkarni-Nomizu tensor field so that it satisfies%
\begin{equation}
T\left( X,Y,Z,W\right) =-T\left( Y,X,Z,W\right) ,  \label{eq-KN-T-1}
\end{equation}%
\begin{equation}
T\left( X,Y,Z,W\right) =-T\left( X,Y,W,Z\right) ,  \label{eq-KN-T-2}
\end{equation}%
\begin{equation}
T\left( X,Y,Z,W\right) =T\left( Z,W,X,Y\right) ,  \label{eq-KN-T-3}
\end{equation}%
\begin{equation}
T\left( X,Y,Z,W\right) +T\left( Y,Z,X,W\right) +T\left( Z,X,Y,W\right) =0,
\label{eq-KN-T-4}
\end{equation}%
\begin{equation}
T\left( X,Y,Z,W\right) +T\left( X,Z,W,Y\right) +T\left( X,W,Y,Z\right) =0
\label{eq-KN-T-5}
\end{equation}%
for all vector fields $X$, $Y$, $Z$ and $W$ on $M$. It is observed that if $%
T $ satisfies any two of the three conditions (\ref{eq-KN-T-1}), (\ref%
{eq-KN-T-2}), (\ref{eq-KN-T-3}) and any one of the two conditions (\ref%
{eq-KN-T-4}), (\ref{eq-KN-T-5}), then it also satisfies the remainning two
relations.

\begin{definition}
{\rm (\cite[Kobayashi and Nomizu 1963, p. 209]{Kobayashi_Nomizu_1963}, \cite[%
Takahashi 1972]{Takahashi_72_KJSM})} A Riemannian manifold $\left(
M,g\right) $ with constant sectional curvature $c$ is called a real space
form, and its Riemann-Christoffel curvature tensor field $R^{M}$ is given by 
{\rm \cite[Tripathi 2026, p. 29]{Tripathi_26_WS}} 
\begin{equation}
R^{M}=\frac{c}{2}\left( g\circledast g\right) .  \label{eq-RSF}
\end{equation}
\end{definition}

\begin{definition}
{\rm \cite[Ogiue 1972]{Ogiue_72_JMSJ}} Let $M$ be an almost Hermitian
manifold with an almost Hermitian structure $\left( J,g\right) $. Then $M$
becomes a K\"{a}hler manifold if $\nabla J=0$. A k\"{a}hler manifold with
constant holomorphic sectional curvature $c$ is called a complex space form $%
M\left( c\right) $, and its Riemann-Christoffel curvature tensor field is
given by {\rm \cite[Tripathi 2026, p. 79]{Tripathi_26_WS}} 
\begin{equation}
R^{M}=\frac{c}{8}\left( g\circledast g\right) +\frac{c}{4}\left\{ \frac{1}{2}%
\left( J^{b}\circledast J^{b}\right) -\left( J^{b}\circledcirc J^{b}\right)
\right\} ,  \label{eq-GCSF}
\end{equation}%
where $J^{\flat }$ denotes the $(0,2)$-tensor field associated with the $%
(1,1)$-tensor field $J$, defined by 
\[
J^{\flat }(X,Y)=g(X,JY) 
\]%
for all vector fields $X$, $Y$ on $M$. Moreover, for any vector field $X$ on 
$M$, we write 
\begin{equation}
JX=PX+QX,  \label{decompose_GCSF_RM}
\end{equation}%
where $PX\in {\cal H}$, $QX\in {\cal V}$ such that 
\[
\left\Vert Q\right\Vert ^{2}=\sum\limits_{i=1}^{\ell }\Vert QV_{i}\Vert
^{2}=\sum\limits_{i,j=1}^{\ell }\left( g(QV_{i},V_{j})\right) ^{2},\quad
\left\Vert P\right\Vert ^{2}=\sum\limits_{i=1}^{s}\Vert PU_{i}\Vert
^{2}=\sum\limits_{i,j=1}^{s}\left( g(PU_{i},U_{j})\right) ^{2}, 
\]%
\[
\left\Vert P^{{\cal V}}\right\Vert ^{2}=\sum\limits_{i=1}^{\ell }\Vert
PV_{i}\Vert ^{2}=\sum\limits_{i=1}^{\ell }\sum\limits_{j=1}^{s}\left(
g(PV_{i},U_{j})\right) ^{2}. 
\]
\end{definition}

\begin{definition}
{\rm (\cite[Alegre et al. 2004]{Alegre_Blair_Carriazo_04_IJM}, \cite[Blair
2010]{Blair_10_BB}, \cite[Vanhecke and Janssens 1981]{Vanhecke_Janssens_81_KMJ})} An almost contact metric manifold $M$ equipped with an
almost contact metric structure $\left( \varphi ,\xi ,\eta ,g\right) $ is
called a generalized Sasakian space form, denoted by $M\left(
c_{1},c_{2},c_{3}\right) $, if its Riemann-Christoffel curvature tensor
field $R^{M}$ satisfies {\rm \cite[Tripathi 2026, p. 115]{Tripathi_26_WS}} 
\begin{equation}
R^{M}=\frac{1}{2}c_{1}\left( g\circledast g\right) +c_{2}\left\{ \frac{1}{2}%
\left( \varphi ^{b}\circledast \varphi ^{b}\right) -\left( \varphi
^{b}\circledcirc \varphi ^{b}\right) \right\} -c_{3}\left( g\circledast
\left( \eta \otimes \eta \right) \right) ,  \label{eq-GSSF}
\end{equation}%
where $\varphi ^{\flat }$ denotes the $(0,2)$-tensor field associated with
the $(1,1)$-tensor field $\varphi $, defined by 
\[
\varphi ^{\flat }(X,Y)=g(X,\varphi Y) 
\]%
for all vector fields $X,Y$ on $M$. Moreover, for any vector field $X$ on $M$, we write 
\begin{equation}
\varphi X=PX+QX,  \label{decompose_GSSF_RM}
\end{equation}%
where $PX\in {\cal H}$, $QX\in {\cal V}$ such that 
\[
\left\Vert Q\right\Vert ^{2}=\sum\limits_{i=1}^{\ell }\Vert QV_{i}\Vert
^{2}=\sum\limits_{i,j=1}^{\ell }\left( g(QV_{i},V_{j})\right) ^{2},\quad
\left\Vert P\right\Vert ^{2}=\sum\limits_{i=1}^{s}\Vert PU_{i}\Vert
^{2}=\sum\limits_{i,j=1}^{s}\left( g(PU_{i},U_{j})\right) ^{2}, 
\]%
\[
\left\Vert P^{{\cal V}}\right\Vert ^{2}=\sum\limits_{i=1}^{\ell }\Vert
PV_{i}\Vert ^{2}=\sum\limits_{i=1}^{\ell }\sum\limits_{j=1}^{s}\left(
g(PV_{i},U_{j})\right) ^{2}. 
\]
In particular, we also have the following notions \cite{Alegre_Blair_Carriazo_04_IJM,Blair_10_BB,Vanhecke_Janssens_81_KMJ} \newline
\begin{table}[h]
\begin{tabular}{lcc}
\hline
{\bf Space form $M_{1}(c)$} & {\bf Value of $c_{1}$} & {\bf Values of $%
c_{2},c_{3}$} \\ 
Sasakian & $\frac{c+3}{4}$ & $\frac{c-1}{4}$ \\ 
Kenmotsu & $\frac{c-3}{4}$ & $\frac{c+1}{4}$ \\ 
cosymplectic & $\frac{c}{4}$ & $\frac{c}{4}$ \\ 
$C(\alpha )$ & $\frac{c+3\alpha}{4}$ & $\frac{c-\alpha }{4}$ \\ 
\hline
\end{tabular}
\centering
\caption{Particular cases of generalised Sasakian space form}
\label{Table 1}
\end{table}
\end{definition}

\begin{lemma}
\label{Lemma_Tripathi} {\rm \cite[Tripathi 2017]{Tripathi_17_NM}} Let $%
\Lambda =\{(t_{1},\dots ,t_{n})\in {\Bbb R}^{n}:t_{1}+\cdots +t_{n}=k\}$ be
a hyperplane of ${\Bbb R}^{n}$, and $f:{\Bbb R}^{n}\rightarrow {\Bbb R}$ be
a quadratic form given by 
\[
f\left( t_{1},\dots ,t_{n}\right) =\lambda
_{1}\sum_{i=1}^{n-1}t_{i}^{2}+\lambda _{2}t_{n}^{2}-2\sum_{1\leq i<j\leq
n}t_{i}t_{j},\quad \lambda _{1},\lambda _{2}\in {\Bbb R}^{+}. 
\]%
Then the constrained extremum problem $\min\limits_{(t_{1},\dots ,t_{n})\in
\Lambda }f$ has the global solution%
\[
t_{1}=t_{2}=\cdots =t_{n-1}=\frac{k}{\lambda _{1}+1},\quad t_{n}=\frac{k}{%
\lambda _{2}+1}=\frac{k\left( n-1\right) }{\left( \lambda _{1}+1\right)
\lambda _{2}}=\frac{k\left( \lambda _{1}-n+2\right) }{\lambda _{1}+1}, 
\]%
provided that $\lambda _{2}=\frac{n-1}{\lambda _{1}-n+2}$.
\end{lemma}

We begin by fixing some notation that will be used throughout the paper. The scalar curvatures defined by%
\begin{equation}
2{\tau }_{{\cal V}}^{{\cal V}}\left( p\right) =\sum\limits_{i,j=1}^{\ell }R^{%
{\cal V}}\left( V_{i},V_{j},V_{j},V_{i}\right) ,\quad 2{\tau }_{{\cal V}%
}^{M_{1}}\left( p\right) =\sum\limits_{i,j=1}^{\ell }R^{M_{1}}\left(
V_{i},V_{j},V_{j},V_{i}\right) ,  \label{eq-scalV}
\end{equation}%
\begin{equation}
2\tau _{{\cal H}}^{{\cal H}}\left( p\right) =\sum\limits_{i,j=1}^{s}R^{{\cal %
H}}\left( U_{i},U_{j},U_{j},U_{i}\right) ,\ 2\tau _{{\cal H}}^{M_{1}}\left(
p\right) =\sum\limits_{i,j=1}^{s}R^{M_{1}}\left(
U_{i},U_{j},U_{j},U_{i}\right) .  \label{eq-Hscal}
\end{equation}%
Consequently, we define the normalized scalar curvatures as
\begin{equation}
\rho _{{\cal V}}^{{\cal V}}=\frac{2{\tau }_{{\cal V}}^{{\cal V}}\left(
p\right) }{\ell \left( \ell -1\right) },\quad \rho _{{\cal V}}^{M_{1}}=\frac{%
2{\tau }_{{\cal V}}^{M_{1}}\left( p\right) }{\ell \left( \ell -1\right) },
\label{eq-Nscal}
\end{equation}%
\begin{equation}
\rho _{{\cal H}}^{{\cal H}}=\frac{2\tau _{{\cal H}}^{{\cal H}}\left(
p\right) }{s\left( s-1\right) },\quad \rho _{{\cal H}}^{M_{1}}=\frac{2\tau _{%
{\cal H}}^{M_{1}}\left( p\right) }{s\left( s-1\right) }.  \label{eq-Nhscal}
\end{equation}%
Also, we set 
\begin{eqnarray}
\left( {\cal T}^{{\cal H}}\right) _{ij}^{\alpha } &=&g\left( {\cal T}%
_{V_{i}}^{{\cal H}}V_{j},U_{\alpha }\right) ,\quad i,j=1,\dots ,\ell ,\quad
\alpha =1,\dots ,s,  \nonumber \\
\left\Vert {\cal T}^{{\cal H}}\right\Vert ^{2} &=&\sum_{i,j=1}^{\ell
}g_{1}\left( {\cal T}_{V_{i}}^{{\cal H}}V_{j},{\cal T}_{V_{i}}^{{\cal H}%
}V_{j}\right) ,\quad {\rm trace~}{\cal T}^{{\cal H}}=\sum_{i=1}^{\ell }{\cal %
T}_{V_{i}}^{{\cal H}}V_{i},  \nonumber \\
\left\Vert {\rm trace~\,}{\cal T}^{{\cal H}}\right\Vert ^{2} &=&g_{1}\left( 
{\rm trace}~{\cal T}^{{\cal H}},{\rm trace}~{\cal T}^{{\cal H}}\right) ,
\label{eq-Notvert}
\end{eqnarray}%
where
\[
{\cal T}^{\cal H}_{V_{i}}V_{j}=h\nabla_{V_{i}}V_{j}
\]
\begin{equation}
\left\Vert {\cal T}^{{\cal V}}\right\Vert ^{2}=\sum_{j=1}^{\ell
}\sum_{i=1}^{s}g_{1}\left( {\cal T}_{V_{j}}^{{\cal H}}U_{i},{\cal T}%
_{V_{j}}^{{\cal H}}U_{i}\right) ,\quad \left\Vert {\cal A}^{{\cal H}%
}\right\Vert ^{2}=\sum_{j=1}^{\ell }\sum_{i=1}^{s}g_{1}\left( {\cal A}%
_{U_{i}}^{{\cal H}}V_{j},{\cal A}_{U_{i}}^{{\cal H}}V_{j}\right) ,
\label{eq-TVAH}
\end{equation}%
\begin{eqnarray}
\left( {\cal A}^{{\cal V}}\right) _{ij}^{\alpha } &=&g_{1}\left( {\cal A}%
_{U_{i}}^{{\cal V}}U_{j},V_{\alpha }\right) ,\quad i,j=1,\dots ,s,\quad
\alpha =1,\dots ,\ell ,  \nonumber \\
\left\Vert {\cal A}^{{\cal V}}\right\Vert ^{2}
&=&\sum_{i,j=1}^{s}g_{1}\left( {\cal A}^{{\cal V}}_{U_{i}}U_{j},{\cal A}^{{\cal V}}%
_{U_{i}}U_{j}\right) ,\quad {\rm trace~}{\cal A}^{{\cal V}}=\sum_{i=1}^{s}%
{\cal A}^{{\cal V}}_{U_{i}}U_{i},  \nonumber \\
\left\Vert {\rm trace~\,}{\cal A}^{{\cal V}}\right\Vert ^{2} &=&g_{1}\left( 
{\rm trace}~{\cal A}^{{\cal V}},{\rm trace}~{\cal A}^{{\cal V}}\right) .
\label{eq-Nothor}
\end{eqnarray}%
where
\[
{\cal A}^{\cal V}_{U_{i}}U_{j}=\frac{1}{2}v[U_{i},U_{j}].
\]
The horizontal divergence of any vector field $U$ on ${\cal H}$ is given by 
\cite{Gulbahar_Meric_Kilic_17_KJM} 
\[
\breve{\delta}\left( U\right) =\sum_{i=1}^{s}g_{1}\left( \nabla
_{U_{i}}U,U_{i}\right) ,
\]%
then, we have 
\begin{equation}
\breve{\delta}\left( N\right) =\sum_{j=1}^{\ell }\sum_{i=1}^{s}g_{1}\left(
\left( \nabla _{U_{i}}{\cal T}\right) \left( V_{j},V_{j}\right)
,U_{i}\right) .  \label{eq-delta-N}
\end{equation}%
The {\it Casorati curvatures }$C^{{\cal V}}$ of the vertical space and $C^{%
{\cal H}}$\ of horizontal space are defined as 
\begin{equation}
C^{{\cal V}}=\frac{1}{\ell }\left\Vert {\cal T}^{{\cal H}}\right\Vert ^{2}=%
\frac{1}{\ell }\sum_{\alpha =1}^{s}\sum_{i,j=1}^{\ell }\left( \left( {\cal T}%
^{{\cal H}}\right) _{ij}^{\alpha }\right) ^{2},\quad C^{{\cal H}}=\frac{1}{s}%
\left\Vert {\cal A}^{{\cal V}}\right\Vert ^{2}=\frac{1}{s}\sum_{\alpha
=1}^{\ell }\sum_{i,j=1}^{s}\left( \left( {\cal A}^{{\cal V}}\right)
_{ij}^{\alpha }\right) ^{2}.  \label{eq-CasCurv}
\end{equation}%
Let $L^{{\cal V}}$ and $L^{{\cal H}}$ be $k$-dimensional ($k\geq 2$)
subspaces of the vertical and horizontal distributions with orthonormal
bases $\{V_{1},\dots ,V_{k}\}$ and $\{U_{1},\dots ,U_{k}\}$, respectively.
Then the Casorati curvatures $C^{L^{{\cal V}}}$ and $C^{L^{{\cal H}}}$ are
given by 
\begin{equation}
C^{L^{{\cal V}}}=\frac{1}{k}\sum_{\alpha =1}^{s}\sum_{i,j=1}^{k}\left(
\left( {\cal T}^{{\cal H}}\right) _{ij}^{\alpha }\right) ^{2},
\label{eq-CLV}
\end{equation}
\begin{equation}
C^{L^{{\cal H}}}=\frac{1}{k}\sum_{\alpha =1}^{\ell }\sum_{i,j=1}^{k}\left(
\left( {\cal A}^{{\cal V}}\right) _{ij}^{\alpha }\right) ^{2}.
\label{eq-CLH}
\end{equation}%
Moreover, the normalized $\delta ^{{\cal V}}${\it -Casorati curvatures} $%
\delta _{C}^{{\cal V}}(\ell -1)$ and $\hat{\delta}_{C}^{{\cal V}}(\ell -1)$
associated with the vertical space at a point $p$ are given by 
\begin{equation}
\left[ \delta _{C}^{{\cal V}}(\ell -1)\right] _{p}=\frac{1}{2}C_{p}^{{\cal V}%
}+\frac{\left( \ell +1\right) }{2\ell }\inf \{C^{L^{{\cal V}}}|L^{{\cal V}}\ 
{\rm a\ hyperplane\ of\ }(\ker \pi _{\ast {p}})\},  \label{eq-delta-ver-1}
\end{equation}%
and 
\begin{equation}
\left[ \hat{\delta}_{C}^{{\cal V}}(\ell -1)\right] _{p}=2C_{p}^{{\cal V}}-%
\frac{\left( 2\ell -1\right) }{2\ell }\sup \{C^{L^{{\cal V}}}|L^{{\cal V}}\ 
{\rm a\ hyperplane\ of\ }(\ker \pi _{\ast {p}})\}.  \label{eq-delta-ver-2}
\end{equation}%
Furthermore, the normalized $\delta ^{{\cal H}}${\it -Casorati curvatures} $%
\delta _{C}^{{\cal H}}(s-1)$ and $\hat{\delta}_{C}^{{\cal H}}(s-1)$
associated with the vertical space at a point $p$ are given by 
\begin{equation}
\left[ \delta _{C}^{{\cal H}}(s-1)\right] _{p}=\frac{1}{2}C_{p}^{{\cal H}}+%
\frac{\left( s+1\right) }{2s}\inf \{C^{L^{{\cal H}}}|L^{{\cal H}}\ {\rm a\
hyperplane\ of\ }{\cal H}_{p}\},  \label{eq-delta-hor-1}
\end{equation}%
and 
\begin{equation}
\left[ \hat{\delta}_{C}^{{\cal V}}(s-1)\right] _{p}=2C_{p}^{{\cal H}}-\frac{%
\left( 2s-1\right) }{2s}\sup \{C^{L^{{\cal H}}}|L^{{\cal H}}\ {\rm a\
hyperplane\ of\ }{\cal H}_{p}\}.  \label{eq-delta-hor-2}
\end{equation}

\section{Casorati inequalities for Riemannian submersion along mixed
distributions \label{section 3}}

We start with the following theorem which provide relation between the normalised scalar curvature and normalised Casorati curvature of vertical and horizontal distributions for Riemannian submersions between Riemannian manifolds:

\begin{theorem}
\label{Th GEN CAS}Let $\pi :(M_{1},g_{1})\rightarrow \left(
M_{2},g_{2}\right) $ be a Riemannian submersion between two Riemannian
manifolds with $\dim {\cal H}_{p}=s\geq 3$ and $\dim {\cal V}_{p}=\ell \geq
3 $, 
\begin{eqnarray}
\frac{\rho _{{\cal H}}^{{\cal H}}}{\ell (\ell -1)}+\frac{\rho _{{\cal V}}^{%
{\cal V}}}{s(s-1)} &\leq &\frac{1}{s(s-1)}\delta _{C}^{{\cal V}}\left( \ell
-1\right) +\frac{1}{\ell (\ell -1)}\delta _{C}^{{\cal H}}\left( s-1\right) 
\nonumber \\
&&+\left( \frac{2\tau ^{M_{1}}\left( p\right) +2\breve{\delta}\left(
N\right) -\left\Vert {\cal T}^{{\cal V}}\right\Vert ^{2}+\left\Vert {\cal A}%
^{{\cal H}}\right\Vert ^{2}}{s(s-1)\ell (\ell -1)}\right) ,
\label{Gen Ine Cas 1}
\end{eqnarray}%
and%
\begin{eqnarray}
\frac{\rho _{{\cal H}}^{{\cal H}}}{\ell (\ell -1)}+\frac{\rho _{{\cal V}}^{%
{\cal V}}}{s(s-1)} &\leq &\frac{1}{s(s-1)}\hat{\delta}_{C}^{{\cal V}}\left(
\ell -1\right) +\frac{1}{\ell (\ell -1)}\hat{\delta}_{C}^{{\cal H}}\left(
s-1\right)  \nonumber \\
&&+\left( \frac{2\tau ^{M_{1}}\left( p\right) +2\breve{\delta}\left(
N\right) -\left\Vert {\cal T}^{{\cal V}}\right\Vert ^{2}+\left\Vert {\cal A}%
^{{\cal H}}\right\Vert ^{2}}{s(s-1)\ell (\ell -1)}\right) ,
\label{Gen Ine Cas 2}
\end{eqnarray}%
where $\tau ^{M_{1}}\left( p\right) =\tau _{{\cal V}}^{M_{1}}\left( p\right) +\tau
_{{\cal H}}^{M_{1}}\left( p\right) +\sum_{i=1}^{s}\sum_{j=1}^{\ell
}R^{M_{1}}\left( U_{i},V_{j},V_{j},U_{i}\right)$. The equality holds in any of the above two inequalities at a point $p\in ${$%
M $}${_{1}}$ if and only if 
\begin{equation}
\left( {\cal T}^{{\cal H}}\right) _{11}^{{\cal \alpha }}=\left( {\cal T}^{%
{\cal H}}\right) _{22}^{{\cal \alpha }}=\cdots =\left( {\cal T}^{{\cal H}%
}\right) _{\ell -1\,\ell -1}^{{\cal \alpha }}=\frac{1}{2}\left( {\cal T}^{%
{\cal H}}\right) _{\ell \ell }^{{\cal \alpha }},  \label{eq-Cas-equality-1}
\end{equation}%
\begin{equation}
\left( {\cal T}^{{\cal H}}\right) _{ij}^{\alpha }=0,\quad 1\leq i\neq j\leq
\ell ,  \label{eq-Cas-equality-2}
\end{equation}%
\begin{equation}
\left( {\cal A}^{{\cal V}}\right) _{ij}^{^{\alpha }}=0,\quad 1\leq i,j\leq s.
\label{eq-Cas-equality-3}
\end{equation}%
The equality conditions are interpreted as follows. The first condition states that
$g_{1}(V_{1},{\cal T}_{V_{1}}^{{\cal V}}U_{\alpha })=g_{1}(V_{2},{\cal %
T}_{V_{2}}^{{\cal V}}U_{\alpha })=\cdots =g_{1}(V_{s-1},{\cal T}_{V_{s-1}}^{%
{\cal V}}U_{\alpha })=\frac{1}{2}g_{1}(V_{\ell },{\cal T}_{V_{\ell }}^{{\cal %
V}}U_{\alpha })$ with respect to all horizontal directions $(U_{\alpha },%
\text{where}~\alpha \in \{1,\dots ,s\})$. Equivalently, there exist $s$
mutually orthogonal horizontal unit vector fields such that the shape
operator in all directions possesses an eigenvalue of multiplicity $(\ell -1)$ and that for each $U_{\alpha }$ the distinguished eigendirections
are identical, specifically $V_{\ell }$. Therefore, the leaves of the vertical space
(called fibers of $\pi $) are invariantly quasi-umbilical {\rm \cite%
{Decu_Haesen_Verstraelen_08_JIPAM}}. The second condition gives $g_{1}(V_{j},%
{\cal T}_{V_{i}}^{{\cal V}}U_{\alpha })=0$ with respect to all horizontal
directions $(U_{\alpha },\text{where}~\alpha \in \{1,\dots ,s\})$.
Equivalently, the shape operator matrices are diagonal, which implies that they commute.
The third condition indicates that ${\cal A}$ vanishes on the horizontal space;
From a geometric perspective, this implies that the horizontal space is integrable.
\end{theorem}

\begin{proof} The scalar curvature $\tau ^{M_{1}}$ of $M_{1}$ is given by \cite{Singh_Meena_Meena_25_Arxiv}
\begin{eqnarray}
\tau ^{M_{1}}\left( p\right) &=&\sum_{1\leq i<j\leq \ell }R^{M_{1}}\left(
V_{i},V_{j},V_{j},V_{i}\right) +\sum_{1\leq i<j\leq s}R^{M_{1}}\left(
U_{i},U_{j},U_{j},U_{i}\right)  \nonumber \\
&&+\sum_{i=1}^{s}\sum_{j=1}^{\ell }R^{M_{1}}\left(
U_{i},V_{j},V_{j},U_{i}\right) .  \label{eq-tau-N1}
\end{eqnarray}%
From (\ref{eq-scalV}), (\ref{eq-Hscal}) and (\ref{eq-tau-N1}), we obtain 
\begin{equation}
\tau ^{M_{1}}\left( p\right) =\tau _{{\cal V}}^{M_{1}}\left( p\right) +\tau
_{{\cal H}}^{M_{1}}\left( p\right) +\sum_{i=1}^{s}\sum_{j=1}^{\ell
}R^{M_{1}}\left( U_{i},V_{j},V_{j},U_{i}\right) .  \label{eq-scalN1}
\end{equation}%
Using (\ref{Gauss_Codazz_RS_T}), (\ref{Gauss_Codazz_RS_A}), (\ref%
{eq-P-(12)}), \ref{eq-scalV}, \ref{eq-Hscal}, (\ref{eq-Notvert}) and (\ref{eq-CasCurv}) in (\ref{eq-scalN1}%
), we obtain 
\begin{eqnarray}
2\tau ^{M_{1}}\left( p\right) &=&2\tau _{{\cal H}}^{{\cal H}}\left( p\right)
+2\tau _{{\cal V}}^{{\cal V}}\left( p\right) -\left\Vert {\rm trace}{\cal T}%
^{{\cal H}}\right\Vert ^{2}+\ell C^{{\cal V}}+3sC^{{\cal H}}  \nonumber \\
&&-\left\Vert {\rm trace}{\cal A}^{{\cal V}}\right\Vert
^{2}-2\sum_{j=1}^{\ell }\sum_{i=1}^{s}g_{1}\left( \left( \nabla _{U_{i}}%
{\cal T}\right) \left( V_{j},V_{j}\right) ,U_{i}\right)  \nonumber \\
&&+\sum_{j=1}^{\ell }\sum_{i=1}^{s}\left\{ g_{1}\left( {\cal T}_{V_{j}}^{%
{\cal V}}U_{i},{\cal T}_{V_{j}}^{{\cal V}}U_{i}\right) -g_{1}\left( {\cal A}%
_{U_{i}}^{{\cal H}}V_{j},{\cal A}_{U_{i}}^{{\cal H}}V_{j}\right) \right\} .
\label{eq-scalN1-2}
\end{eqnarray}%
Substituting (\ref{eq-TVAH}) and (\ref{eq-delta-N}) into (\ref{eq-scalN1-2}), we
obtain 
\begin{eqnarray}
2\tau ^{M_{1}}\left( p\right) &=&2\tau _{{\cal H}}^{{\cal H}}\left( p\right)
+2\tau _{{\cal V}}^{{\cal V}}\left( p\right) -\left\Vert {\rm trace}{\cal T}%
^{{\cal H}}\right\Vert ^{2}+\ell C^{{\cal V}}+3sC^{{\cal H}}  \nonumber \\
&&-\left\Vert {\rm trace}{\cal A}^{{\cal V}}\right\Vert ^{2}-2\breve{\delta}%
\left( N\right) +\left\Vert {\cal T}^{{\cal V}}\right\Vert ^{2}-\left\Vert 
{\cal A}^{{\cal H}}\right\Vert ^{2}.  \label{eq-scalN1-3}
\end{eqnarray}%
Now, we consider the quadratic polynomial with respect to the components of $%
{\cal T}$ and ${\cal A}$, 
\begin{eqnarray}
{\cal P}^{{\cal HV}} &=&\frac{1}{2}\ell (\ell -1){C}^{{\cal V}}+\frac{1}{2}%
s(s-1){C}^{{\cal H}}+\frac{1}{2}(\ell ^{2}-1){C}^{L^{{\cal V}}}  \nonumber \\
&&+\frac{1}{2}(s^{2}-1){C}^{L^{{\cal H}}}+2\tau ^{M_{1}}\left( p\right) +2%
\breve{\delta}\left( N\right) -\left\Vert {\cal T}^{{\cal V}}\right\Vert
^{2}+\left\Vert {\cal A}^{{\cal H}}\right\Vert ^{2}  \nonumber \\
&&-2\tau _{{\cal H}}^{{\cal H}}\left( p\right) -2\tau _{{\cal V}}^{{\cal V}%
}\left( p\right) .  \label{eq-qudpol-4}
\end{eqnarray}%
From (\ref{eq-scalN1-3}) and (\ref{eq-qudpol-4}), we obtain%
\begin{eqnarray}
{\cal P}^{{\cal HV}} &=&\frac{1}{2}\ell (\ell -1){C}^{{\cal V}}+\frac{1}{2}%
s(s-1){C}^{{\cal H}}+\frac{1}{2}(\ell ^{2}-1){C}^{L^{{\cal V}}}  \nonumber \\
&&+\frac{1}{2}(s^{2}-1){C}^{L^{{\cal H}}}-\left\Vert {\rm trace}{\cal T}^{%
{\cal H}}\right\Vert ^{2}+\ell C^{{\cal V}}  \nonumber \\
&&-\left\Vert {\rm trace}{\cal A}^{{\cal V}}\right\Vert ^{2}+3sC^{{\cal H}}.
\label{eq-qudpol-5}
\end{eqnarray}%
Assume the hyperplanes $L^{{\cal V}}$ and $L^{{\cal H}}$ are spanned by $%
\{V_{1},\dots ,V_{\ell -1}\}$ and $\left\{ U_{1},\ldots ,U_{s-1}\right\} $,
respectively. Then, using (\ref{eq-Notvert}), (\ref{eq-Nothor}), (\ref{eq-CLV}) and (\ref%
{eq-CLH}) in (\ref{eq-qudpol-5}), we obtain%
\begin{eqnarray*}
{\cal P}^{{\cal HV}} &=&\sum_{\alpha =1}^{s}\sum_{i=1}^{\ell -1}\left\{ \ell
\left( \left( {\cal T}^{{\cal H}}\right) _{ii}^{\alpha }\right) ^{2}+(\ell
+1)\left( \left( {\cal T}^{{\cal H}}\right) _{i\ell }^{\alpha }\right)
^{2}\right\} \\
&&+\sum_{\alpha =1}^{s}\left\{ 2\left( \ell +1\right) \sum_{1=i<j}^{\ell
-1}\left( \left( {\cal T}^{{\cal H}}\right) _{ij}^{\alpha }\right)
^{2}-2\sum_{1=i<j}^{\ell }\left( {\cal T}^{{\cal H}}\right) _{ii}^{\alpha
}\left( {\cal T}^{{\cal H}}\right) _{jj}^{{\cal \alpha }}+\frac{(\ell -1)}{2}%
\left( \left( {\cal T}^{{\cal H}}\right) _{\ell \ell }^{\alpha }\right)
^{2}\right\} \\
&&+\sum_{\alpha =1}^{\ell }\left\{ (s+3)\sum_{i,j=1}^{s-1}\left( \left( 
{\cal A}^{{\cal V}}\right) _{ii}^{\alpha }\right) ^{2}+\frac{1}{2}%
(s+5)\left( \left( {\cal A}^{{\cal V}}\right) _{ss}^{\alpha }\right)
^{2}+2(s+3)\sum_{1\leq i<j\leq s-1}\left( \left( {\cal A}^{{\cal V}}\right)
_{ij}^{\alpha }\right) ^{2}\right. \\
&&\left. +(s+5)\sum_{i=1}^{s-1}\left( \left( {\cal A}^{{\cal V}}\right)
_{is}^{\alpha }\right) ^{2}-\sum_{i,j=1}^{s}\left( {\cal A}^{{\cal V}%
}\right) _{ii}^{\alpha }\left( {\cal A}^{{\cal V}}\right) _{jj}^{\alpha
}\right\} .
\end{eqnarray*}%
Thus, 
\begin{eqnarray}
{\cal P}^{{\cal HV}} &\geq &\sum_{\alpha =1}^{s}\left( \ell \sum_{i=1}^{\ell
-1}\left( \left( {\cal T}^{{\cal H}}\right) _{ii}^{\alpha }\right) ^{2}+%
\frac{(\ell -1)}{2}\left( \left( {\cal T}^{{\cal H}}\right) _{\ell \ell
}^{\alpha }\right) ^{2}-2\sum_{1=i<j}^{\ell }\left( {\cal T}^{{\cal H}%
}\right) _{ii}^{\alpha }\left( {\cal T}^{{\cal H}}\right) _{jj}^{\alpha
}\right)  \nonumber \\
&&+\sum_{\alpha =1}^{\ell }\left\{ (s+3)\sum_{i,j=1}^{s-1}\left( \left( 
{\cal A}^{{\cal V}}\right) _{ii}^{\alpha }\right) ^{2}+\frac{(s+5)}{2}\left(
\left( {\cal A}^{{\cal V}}\right) _{ss}^{\alpha }\right)
^{2}-\sum_{i,j=1}^{s}\left( {\cal A}^{{\cal V}}\right) _{ii}^{\alpha }\left( 
{\cal A}^{{\cal V}}\right) _{jj}^{\alpha }\right\} .  \label{eq-phv-ineq}
\end{eqnarray}%
Since {\cal A} is skew symmetric, we have $\left( {\cal A}^{{\cal V}}\right)
_{ii}^{\alpha }=0$ for all $i=1,\ldots ,s$, $\alpha =1,\ldots ,\ell $.
Therefore, from (\ref{eq-phv-ineq}), we obtain 
\[
{\cal P}^{{\cal HV}}\geq \sum_{\alpha =1}^{s}\left( \ell \sum_{i=1}^{\ell
-1}\left( \left( {\cal T}^{{\cal H}}\right) _{ii}^{\alpha }\right) ^{2}+%
\frac{(\ell -1)}{2}\left( \left( {\cal T}^{{\cal H}}\right) _{\ell \ell
}^{\alpha }\right) ^{2}-2\sum_{1=i<j}^{\ell }\left( {\cal T}^{{\cal H}%
}\right) _{ii}^{\alpha }\left( {\cal T}^{{\cal H}}\right) _{jj}^{\alpha
}\right) . 
\]%
Let $f:{\Bbb R}^{\ell }\rightarrow {\Bbb R}{\bf \,}$\ be a quadratic form
for each $1\leq \alpha \leq s$ such that 
\[
f(\left( {\cal T}^{{\cal H}}\right) _{11}^{\alpha },\dots ,\left( {\cal T}^{%
{\cal H}}\right) _{\ell \ell }^{\alpha })=\sum_{i=1}^{\ell -1}\ell \left(
\left( {\cal T}^{{\cal H}}\right) _{ii}^{\alpha }\right) ^{2}+\frac{(\ell -1)%
}{2}\left( \left( {\cal T}^{{\cal H}}\right) _{\ell \ell }^{\alpha }\right)
^{2}-2\sum_{1=i<j}^{\ell }\left( {\cal T}^{{\cal H}}\right) _{ii}^{\alpha
}\left( {\cal T}^{{\cal H}}\right) _{jj}^{\alpha }, 
\]%
and constrained extremum problem $\min f$ subject to $\left( {\cal T}^{{\cal %
H}}\right) _{11}^{\alpha }+\cdots +\left( {\cal T}^{{\cal H}}\right) _{\ell
\ell }^{\alpha }=k^{\alpha }\in {\Bbb R}$. Then using Lemma \ref%
{Lemma_Tripathi}, we get the global minimum critical point, 
\[
\left( {\cal T}^{{\cal H}}\right) _{11}^{\alpha }=\cdots =\left( {\cal T}^{%
{\cal H}}\right) _{\ell -1\,\ell -1}^{\alpha }=\frac{k^{\alpha }}{\ell +1}%
,\quad \left( {\cal T}^{{\cal H}}\right) _{\ell \ell }^{\alpha }=\frac{%
2k^{\alpha }}{\ell +1}. 
\]%
Moreover, $f(\left( {\cal T}^{{\cal H}}\right) _{11}^{\alpha },\dots ,\left( 
{\cal T}^{{\cal H}}\right) _{\ell \ell }^{\alpha })=0$. Consequently, we
obtain 
\begin{equation}
{\cal P}^{{\cal HV}}\geq 0.  \label{eq-PHVGT0}
\end{equation}%
From (\ref{eq-qudpol-4}) and (\ref{eq-PHVGT0}), we obtain%
\begin{eqnarray}
2\tau _{{\cal H}}^{{\cal H}}+2\tau _{{\cal V}}^{{\cal V}} &\leq &\frac{1}{2}%
\ell (\ell -1){C}^{{\cal V}}+\frac{1}{2}s(s-1){C}^{{\cal H}}+\frac{1}{2}%
(\ell ^{2}-1){C}^{L^{{\cal V}}}  \nonumber \\
&&+\frac{1}{2}(s^{2}-1){C}^{L^{{\cal H}}}+2\tau ^{M_{1}}\left( p\right) +2%
\breve{\delta}\left( N\right) -\left\Vert {\cal T}^{{\cal V}}\right\Vert
^{2}+\left\Vert {\cal A}^{{\cal H}}\right\Vert ^{2}.  \label{eq-tauh-tauv}
\end{eqnarray}%
Using (\ref{eq-Nscal}) and (\ref{eq-Nhscal}) in (\ref{eq-tauh-tauv}), we
obtain%
\begin{eqnarray}
\frac{\rho _{{\cal H}}^{{\cal H}}}{\ell (\ell -1)}+\frac{\rho _{{\cal V}}^{%
{\cal V}}}{s(s-1)} &\leq &\frac{1}{s(s-1)}\left\{ \frac{1}{2}{C}^{{\cal V}}+%
\frac{\left( \ell +1\right) }{2\ell }{C}^{L^{{\cal V}}}\right\} +\frac{1}{%
\ell (\ell -1)}\left\{ \frac{1}{2}C^{{\cal H}}+\frac{\left( s+1\right) }{2s}{%
C}^{L^{{\cal H}}}\right\}  \nonumber \\
&&+\frac{2\tau ^{M_{1}}\left( p\right) }{s(s-1)\ell (\ell -1)}+\frac{\left( 2%
\breve{\delta}\left( N\right) -\left\Vert {\cal T}^{{\cal V}}\right\Vert
^{2}+\left\Vert {\cal A}^{{\cal H}}\right\Vert ^{2}\right) }{s(s-1)\ell
(\ell -1)}.  \label{eq-rhocurv-1}
\end{eqnarray}%
Similarly, we consider the other quadratic polynomial with respect to the
components of ${\cal T}$ and ${\cal A}$,%
\begin{eqnarray*}
{\cal Q}^{{\cal HV}} &=&2\ell \left( \ell -1\right) {C}^{{\cal V}}+2s\left(
s-1\right) {C}^{{\cal H}}-\frac{1}{2}\left( \ell -1\right) \left( 2\ell
-1\right) {C}^{L^{{\cal V}}}-\frac{1}{2}\left( s-1\right) \left( 2s-1\right) 
{C}^{L^{{\cal H}}} \\
&&+2\tau ^{M_{1}}\left( p\right) +2\breve{\delta}\left( N\right) -\left\Vert 
{\cal T}^{{\cal V}}\right\Vert ^{2}+\left\Vert {\cal A}^{{\cal H}%
}\right\Vert ^{2}-2\tau _{{\cal H}}^{{\cal H}}-2\tau _{{\cal V}}^{{\cal V}}.
\end{eqnarray*}%
Following the similar argument, we arrive at 
\[
{\cal Q}^{{\cal HV}}\geq 0, 
\]%
which implies 
\begin{eqnarray}
\frac{\rho _{{\cal H}}^{{\cal H}}}{\ell (\ell -1)}+\frac{\rho _{{\cal V}}^{%
{\cal V}}}{s(s-1)} &\leq &\frac{1}{s(s-1)}\left\{ 2C_{p}^{{\cal V}}-\frac{%
\left( 2\ell -1\right) }{2\ell }{C}^{L^{{\cal V}}}\right\} +\frac{1}{\ell
(\ell -1)}\left\{ 2C_{p}^{{\cal H}}-\frac{\left( 2s-1\right) }{2s}{C}^{L^{%
{\cal H}}}\right\}  \nonumber \\
&&+\frac{2\tau ^{M_{1}}\left( p\right) }{s(s-1)\ell (\ell -1)}+\frac{\left( 2%
\breve{\delta}\left( N\right) -\left\Vert {\cal T}^{{\cal V}}\right\Vert
^{2}+\left\Vert {\cal A}^{{\cal H}}\right\Vert ^{2}\right) }{s(s-1)\ell
(\ell -1)}.  \label{eq-rhocurv-2}
\end{eqnarray}%
Then (\ref{Gen Ine Cas 1}) follows by taking the infimum on all hyperplane $%
L^{{\cal H}}$ and $L^{{\cal V}}$ in (\ref{eq-rhocurv-1}), while we get (\ref%
{Gen Ine Cas 2}), taking supremum on all hyperplane $L^{{\cal H}}$ and $L^{%
{\cal V}}$ in (\ref{eq-rhocurv-2}).  \end{proof}

We construct two examples to examine the equality cases of Theorem \ref{Th
GEN CAS}: the inequalities obtained in Theorem \ref{Th GEN CAS} do not
attain equality in the first example, whereas they attain equality in the
second.

\begin{example}
Let $\left( M_{1}=\left\{ \left( x^{1},x^{2},x^{3},x^{4},x^{5},x^{6}\right)
\in {\Bbb R}^{6}\ :\ x^{6}>0\right\} ,\quad g_{1}=\left( x^{6}\right)
^{2}\sum_{i=1}^{5}\left( dx^{i}\right) ^{2}+\left( dx^{6}\right) ^{2}\right) 
$ and $\left( M_{2}=\left\{ \left( y^{1},y^{2},y^{3}\right) \in {\Bbb R}%
^{3}:y^{3}>0\right\} ,\quad g_{2}=\left( y^{3}\right) ^{2}\left(
dy^{1}\right) ^{2}+\left( y^{3}\right) ^{2}\left( dy^{2}\right) ^{2}+\left(
dy^{3}\right) ^{2}\right) $ be Riemannian manifold. We define a map $\pi
:\left( M_{1},g_{1}\right) \rightarrow \left( M_{2},g_{2}\right) $ by%
\[
\pi \left( x^{1},x^{2},x^{3},x^{4},x^{5},x^{6}\right) =\left(
x^{4},x^{5},x^{6}\right) . 
\]%
Then, we have%
\begin{eqnarray*}
{\cal V} &=&{\rm span}\left\{ V_{1}=e_{1},V_{2}=e_{2},V_{3}=e_{3}\right\} ,
\\
{\cal H} &=&{\rm span}\left\{ U_{1}=e_{4},U_{2}=e_{5},U_{3}=e_{6}\right\} ,
\\
{\rm range}~\pi _{\ast } &=&{\rm span}\left\{ \pi _{\ast }U_{1}=e_{1}^{\ast
},\pi _{\ast }U_{2}=e_{2}^{\ast },\pi _{\ast }U_{3}=e_{3}^{\ast }\right\} ,
\end{eqnarray*}%
where $\left\{ e_{1}=\frac{1}{x^{6}}\frac{\partial }{\partial x^{1}},e_{2}=%
\frac{1}{x^{6}}\frac{\partial }{\partial x^{2}},e_{3}=\frac{1}{x^{6}}\frac{%
\partial }{\partial x^{3}},e_{4}=\frac{1}{x^{6}}\frac{\partial }{\partial
x^{4}},e_{5}=\frac{1}{x^{6}}\frac{\partial }{\partial x^{5}},e_{6}=\frac{%
\partial }{\partial x^{6}}\right\} $ and \newline
$\left\{ e_{1}^{\ast }=\frac{1}{y^{3}}\frac{\partial }{\partial y^{1}}%
,e_{2}^{\ast }=\frac{1}{y^{3}}\frac{\partial }{\partial y^{2}},e_{3}^{\ast }=%
\frac{\partial }{\partial y^{3}}\right\} $ be orthonormal bases of $%
T_{p}M_{1}$ and $T_{\pi \left( p\right) }M_{2}$, respectively. We observe
that $g_{1}\left( U_{i},U_{j}\right) =g_{2}\left( \pi _{\ast }U_{i},\pi
_{\ast }U_{j}\right) $ for all $U_{i},U_{j}\in {\cal H}$ and ${\rm rank}\pi
_{\ast }=3$. Thus $\pi $ is a Riemannian submersion. The nonzero Christoffel
symbols for $g_{1}$ are%
\[
\Gamma _{i6}^{i}=\frac{1}{x^{6}}\quad and\quad \Gamma
_{ii}^{6}=-x^{6}\quad for\quad 1\leq i\leq 5. 
\]%
The covariant derivative of vertical vector fields are
\[
\nabla _{V_{i}}V_{i}=-\frac{1}{x^{6}}\frac{\partial }{\partial x^{6}}, \quad i=1,2,3, \quad {\rm and} \quad \nabla _{V_{i}}V_{j}=0,\quad i\neq j.
\]
Since $U_{3}=e_{6}=\frac{\partial }{\partial x^{6}}$, it follows that the only non-zero components of ${\cal T}^{{\cal H}}$ are
\begin{equation}\label{comp-exam-1-TH}
\left( {\cal T}^{{\cal H}}\right) _{ii}^{3}=-\frac{1}{x^{6}},\quad 1\leq
i\leq 3.
\end{equation}
Furthermore, using $\left[ U_{i},U_{j}\right] =\left( \nabla _{U_{i}}U_{j}-\nabla _{U_{j}}U_{i}\right)$, we obtain the only non-zero vectors $\left[ U_{i},U_{j}\right]$ are
\[
\left[ U_{1},U_{3}\right] = \left( \frac{1%
}{x^{6}}\right) ^{2}\frac{\partial }{\partial x^{4}}, \quad \left[ U_{2},U_{3}\right] = \left( \frac{1%
}{x^{6}}\right) ^{2}\frac{\partial }{\partial x^{5}},
\]
these vectors belong to {\cal H}, hence
\[
v\left[ U_{i},U_{j}\right]=0, \quad 1\leq i,j\leq 3.
\]
Consequently,
\begin{equation}\label{comp-exam-1-AV}
\left( {\cal A}^{{\cal V}}\right) _{ij}^{\alpha }=0, \quad 1\leq i,j\leq 3,\ 1\leq \alpha \leq 3.
\end{equation}
In view of (\ref{comp-exam-1-TH}) and (\ref{comp-exam-1-AV}), we conclude that the inequality obtained
in Theorem \ref{Th GEN CAS} does not attain equality.
\end{example}

\begin{example}
Let $M_{1}=\left\{ \left( x^{1},x^{2},x^{3},x^{4},x^{5},x^{6}\right) \in 
{\Bbb R}^{6}\ :\ x^{2},x^{4},x^{6}>0\right\} $ and \newline
$M_{2}=\left\{ \left( y^{1},y^{2},y^{3}\right) \in {\Bbb R}^{3}\right\} $ be
two Riemannian manifolds with Riemannian metric $g_{1}=\left( dx^{1}\right)
^{2}+\left( x^{2}\right) ^{2}\left( dx^{2}\right) ^{2}+\left( dx^{3}\right)
^{2}+\left( x^{4}\right) ^{2}\left( dx^{4}\right) ^{2}+\left( dx^{5}\right)
^{2}+\left( x^{6}\right) ^{2}\left( dx^{6}\right) ^{2}$ and $g_{2}=\left(
dy^{1}\right) ^{2}+\left( dy^{2}\right) ^{2}+\left( dy^{3}\right) ^{2}$,
respectively. Define a map $\pi :\left( M_{1},g_{1}\right) \rightarrow
\left( M_{2},g_{2}\right) $ by 
\[
\pi \left( x^{1},x^{2},x^{3},x^{4},x^{5},x^{6}\right) =\left(
x^{1},x^{3},x^{5}\right) . 
\]%
Then, we obtain 
\begin{eqnarray*}
{\cal V} &=&{\rm span}\left\{ V_{1}=e_{2},V_{2}=e_{4},V_{3}=e_{6}\right\} ,
\\
{\cal H} &=&{\rm span}\left\{ U_{1}=e_{1},U_{2}=e_{3},U_{3}=e_{5}\right\} ,
\\
{\rm range}~\pi _{\ast } &=&{\rm span}\left\{ \pi _{\ast }U_{1}=e_{1}^{\ast
},\pi _{\ast }U_{2}=e_{2}^{\ast },\pi _{\ast }U_{3}=e_{3}^{\ast }\right\} ,
\end{eqnarray*}%
where $\left\{ e_{1}=\frac{\partial }{\partial x^{1}},e_{2}=\frac{1}{x^{2}}%
\frac{\partial }{\partial x^{2}},e_{3}=\frac{\partial }{\partial x^{3}}%
,e_{4}=\frac{1}{x^{4}}\frac{\partial }{\partial x^{4}},e_{5}=\frac{\partial 
}{\partial x^{5}},e_{6}=\frac{1}{x^{6}}\frac{\partial }{\partial x^{6}}%
\right\} $ and \newline
$\left\{ e_{1}^{\ast }=\frac{\partial }{\partial y^{1}},e_{2}^{\ast }=\frac{%
\partial }{\partial y^{2}},e_{3}^{\ast }=\frac{\partial }{\partial y^{3}}%
\right\} $ be orthonormal bases of $T_{p}M_{1}$ and $T_{\pi \left( p\right)
}M_{2}$, respectively. We observe that $g_{1}\left( U_{i},U_{j}\right)
=g_{2}\left( \pi _{\ast }U_{i},\pi _{\ast }U_{j}\right) $ for all $%
U_{i},U_{j}\in {\cal H}$ and ${\rm rank}\pi _{\ast }=3$. Thus $\pi $ is a
Riemannian submersion. The nonzero Christoffel symbols for $g_{1}$ are%
\[
\Gamma _{22}^{2}=\frac{1}{x^{2}},\ \Gamma _{44}^{4}=\frac{1}{x^{4}},\ \Gamma
_{66}^{6}=\frac{1}{x^{6}}. 
\]%
The covariant derivative of vertical vector fields are
\[
\nabla_{V_{i}}V_{j}=0, \quad i,j \in \{1,2,3\}.
\]
The components of ${\cal T}^{\cal H}$ are
\[
\left( {\cal T}^{{\cal H}}\right) _{ij}^{\alpha } =0,\quad 1\leq i,j,\alpha \leq 3. 
\]%
Now, we compute the components of the tensor ${\cal A}^{{\cal V}}$ same as
previous example, we get%
\[
\left( {\cal A}^{{\cal V}}\right) _{ij}^{\alpha }=0,\quad 1\leq i,j,\alpha
\leq 3. 
\]%
Hence, we observe that the inequalities obtained in Theorem \ref{Th GEN CAS}
attain equality.
\end{example}

\section{Applications}\label{sec_4}
In this section, we addresses applications of the derived inequalities in Theorem~\ref{Th GEN CAS} to Riemannian submersions whose total manifolds are real, complex, and generalized Sasakian space forms, with the equality cases illustrated by several examples. Direct computations yield the corresponding inequalities for Riemannian submersions whose total spaces are Sasakian, Kenmotsu, cosymplectic, and $C(\alpha)$-space forms.
\begin{theorem}
\label{Theorem RSF} Let $\pi :(M_{1},g_{1})\rightarrow \left(
M_{2},g_{2}\right) $ be a Riemannian submersion from a real space form of
constant sectional curvature $c$ onto a Riemannian manifold with $\dim {\cal %
H}=s\geq 3$, and $\dim {\cal V}=\ell \geq 3$, 
\begin{eqnarray}
\frac{\rho _{{\cal H}}^{{\cal H}}}{\ell (\ell -1)}+\frac{\rho _{{\cal V}}^{%
{\cal V}}}{s(s-1)} &\leq &\frac{1}{s(s-1)}\delta _{C}^{{\cal V}}\left( \ell
-1\right) +\frac{1}{\ell (\ell -1)}\delta _{C}^{{\cal H}}\left( s-1\right) 
\nonumber \\
&&+\frac{c\left( \ell ^{2}+s^{2}+2s\ell -\ell -s\right) }{s(s-1)\ell (\ell
-1)}  \nonumber \\
&&+\frac{\left( 2\breve{\delta}\left( N\right) -\left\Vert {\cal T}^{{\cal V}%
}\right\Vert ^{2}+\left\Vert {\cal A}^{{\cal H}}\right\Vert ^{2}\right) }{%
s(s-1)\ell (\ell -1)},  \label{eq-Cas RSF 1}
\end{eqnarray}%
and%
\begin{eqnarray}
\frac{\rho _{{\cal H}}^{{\cal H}}}{\ell (\ell -1)}+\frac{\rho _{{\cal V}}^{%
{\cal V}}}{s(s-1)} &\leq &\frac{1}{s(s-1)}\hat{\delta}_{C}^{{\cal V}}\left(
\ell -1\right) +\frac{1}{\ell (\ell -1)}\hat{\delta}_{C}^{{\cal H}}\left(
s-1\right)  \nonumber \\
&&+\frac{c\left( \ell ^{2}+s^{2}+2s\ell -\ell -s\right) }{s(s-1)\ell (\ell
-1)}  \nonumber \\
&&+\frac{1}{s(s-1)\ell (\ell -1)}\left( 2\breve{\delta}\left( N\right)
-\left\Vert {\cal T}^{{\cal V}}\right\Vert ^{2}+\left\Vert {\cal A}^{{\cal H}%
}\right\Vert ^{2}\right) .  \label{eq-Cas RSF 2}
\end{eqnarray}%
The equality holds in any of {\rm (\ref{eq-Cas RSF 1})} and {\rm (\ref%
{eq-Cas RSF 2})} at $p\in M_{1}$ if and only if the tensors ${\cal T}$ and $%
{\cal A}$ satisfy {\rm (\ref{eq-Cas-equality-1})}, {\rm (\ref%
{eq-Cas-equality-2})} and {\rm (\ref{eq-Cas-equality-3})}.
\end{theorem}

\begin{proof} Using (\ref{eq-RSF}), (\ref{eq-scalV}) and (\ref%
{eq-Hscal}), we obtain%
\begin{equation}
2\tau _{{\cal V}}^{M_{1}}\left( p\right) =c\ell \left( \ell -1\right) ,\quad
2\tau _{{\cal H}}^{M_{1}}\left( p\right) =cs\left( s-1\right) ,\
\sum_{i=1}^{s}\sum_{j=1}^{\ell }R^{M_{1}}\left(
U_{i},V_{j},V_{j},U_{i}\right) =cs\ell .  \label{eq-scalRSF}
\end{equation}%
From (\ref{eq-scalRSF}), we obtain%
\begin{equation}
2\tau ^{M_{1}}\left( p\right) =c\left( \ell ^{2}+s^{2}+2s\ell -\ell
-s\right) .  \label{eq-NSCALRSF}
\end{equation}%
Substituting (\ref{eq-NSCALRSF}) into (\ref{Gen Ine Cas 1}) and (\ref{Gen Ine Cas 2}), we get (\ref{eq-Cas RSF 1}) and (\ref{eq-Cas RSF 2}), respectively.
\end{proof}

We construct the following example to characterize the equality cases of the
above theorem.

\begin{example}
Let $\left( M_{1}=\left\{ \left( x^{1},x^{2},x^{3},x^{4},x^{5},x^{6}\right)
\in {\Bbb R}^{6}\ \right\},\ g_{1}=\sum_{i=1}^{6}\left( dx^{i}\right)
^{2}\right) $ be a real space form and $\left( M_{2}=\left\{ \left(
y^{1},y^{2},y^{3}\right) \in {\Bbb R}^{3}\ \right\},\ g_{2}=\sum_{i=1}^{3}\left( dy^{i}\right) ^{2}\right) $ be Riemannian
manifold. Define a map $\pi :\left( M_{1},g_{1}\right) \rightarrow \left(
M_{2},g_{2}\right) $ by%
\[
\pi \left( x^{1},x^{2},x^{3},x^{4},x^{5},x^{6}\right) =\left( \frac{%
x^{1}-x^{3}}{\sqrt{2}},x^{4},\frac{x^{5}+x^{6}}{\sqrt{2}}\right) . 
\]%
Then, we obtain%
\begin{eqnarray*}
{\cal V} &=&{\rm span}\left\{ V_{1}=\frac{1}{\sqrt{2}}\frac{\partial }{%
\partial x^{1}}+\frac{1}{\sqrt{2}}\frac{\partial }{\partial x^{3}},V_{2}=%
\frac{\partial }{\partial x^{2}},V_{3}=\frac{1}{\sqrt{2}}\frac{\partial }{%
\partial x^{5}}-\frac{1}{\sqrt{2}}\frac{\partial }{\partial x^{6}}\right\} ,
\\
{\cal H} &=&{\rm span}\left\{ U_{1}=\frac{1}{\sqrt{2}}\frac{\partial }{%
\partial x^{1}}-\frac{1}{\sqrt{2}}\frac{\partial }{\partial x^{3}},U_{2}=%
\frac{\partial }{\partial x^{4}},U_{3}=\frac{1}{\sqrt{2}}\frac{\partial }{%
\partial x^{5}}+\frac{1}{\sqrt{2}}\frac{\partial }{\partial x^{6}}\right\} ,
\\
{\rm range}~\pi _{\ast } &=&{\rm span}\left\{ \pi _{\ast }U_{1}=\frac{%
\partial }{\partial y^{1}},\pi _{\ast }U_{2}=\frac{\partial }{\partial y^{2}}%
,\pi _{\ast }U_{3}=\frac{\partial }{\partial y^{3}}\right\} ,
\end{eqnarray*}%
where $\left\{ \frac{\partial }{\partial x^{1}},\frac{\partial }{\partial
x^{2}},\frac{\partial }{\partial x^{3}},\frac{\partial }{\partial x^{4}},%
\frac{\partial }{\partial x^{5}},\frac{\partial }{\partial x^{6}}\right\} $
and $\left\{ \frac{\partial }{\partial y^{1}},\frac{\partial }{\partial y^{2}%
},\frac{\partial }{\partial y^{3}}\right\} $ are orthonormal bases of $%
T_{p}M_{1}$ and $T_{\pi \left( p\right) }M_{2}$, respectively. It follows that
\[
g_{1}(U_{i},U_{j})
=
g_{2}(\pi_{\ast}U_{i},\pi_{\ast}U_{j}),
\qquad i,j=1,2,3,
\]
and ${\rm rank}\pi_{\ast}=3=\dim M_{2}$.
Therefore, $\pi$ is a Riemannian submersion. Since the metric $g_{1}$
is the standard Euclidean metric on ${\Bbb R}^{6}$, all Christoffel symbols
vanish. Therefore, the O'Neill tensors vanish identically, that is,
\[
\left( {\cal T}^{{\cal H}}\right) _{ij}^{\alpha }=0,\quad \left( {\cal A}^{%
{\cal V}}\right) _{ij}^{\alpha }=0,\quad i,j\in \{1,2,3\},\ \alpha \in
\{1,2,3\}. 
\]%
Therefore, the inequalities established in
Theorem~\ref{Theorem RSF} are satisfied with equality.
\end{example}
\begin{theorem}
\label{Theorem CSF} Let $\pi :(M_{1},g_{1},J)\rightarrow \left(
M_{2},g_{2}\right) $ be a Riemannian submersion from a complex space form of
constant holomorphic sectional curvature $c$ onto a Riemannian manifold such
that $\dim M_{1}=n=2k$ with $\dim {\cal H}_{p}=s\geq 3$ and $\dim {\cal V}%
_{p}=\ell \geq 3$,%
\begin{eqnarray}
\frac{\rho _{{\cal H}}^{{\cal H}}}{\ell (\ell -1)}+\frac{\rho _{{\cal V}}^{%
{\cal V}}}{s(s-1)} &\leq &\frac{1}{s(s-1)}\delta _{C}^{{\cal V}}\left( \ell
-1\right) +\frac{1}{\ell (\ell -1)}\delta _{C}^{{\cal H}}\left( s-1\right) 
\nonumber \\
&&+\frac{1}{s(s-1)\ell (\ell -1)}\frac{c}{4}\left( \ell ^{2}+s^{2}+2sl-\ell
-s\right)  \nonumber \\
&&+\frac{3c}{4\ell \left( \ell -1\right) s(s-1)}\left( \left\Vert
Q\right\Vert ^{2}+\left\Vert P\right\Vert ^{2}+2\left\Vert P^{{\cal V}%
}\right\Vert ^{2}\right)  \nonumber \\
&&+\frac{1}{s(s-1)\ell (\ell -1)}\left( 2\breve{\delta}\left( N\right)
-\left\Vert {\cal T}^{{\cal V}}\right\Vert ^{2}+\left\Vert {\cal A}^{{\cal H}%
}\right\Vert ^{2}\right) ,  \label{eq-CAS CSF 1}
\end{eqnarray}%
and%
\begin{eqnarray}
\frac{\rho _{{\cal H}}^{{\cal H}}}{\ell (\ell -1)}+\frac{\rho _{{\cal V}}^{%
{\cal V}}}{s(s-1)} &\leq &\frac{1}{s(s-1)}\hat{\delta}_{C}^{{\cal V}}\left(
\ell -1\right) +\frac{1}{\ell (\ell -1)}\hat{\delta}_{C}^{{\cal H}}\left(
s-1\right)  \nonumber \\
&&+\frac{c}{4s(s-1)\ell (\ell -1)}\left( \ell ^{2}+s^{2}+2sl-\ell -s\right) 
\nonumber \\
&&+\frac{3c}{4\ell \left( \ell -1\right) s(s-1)}\left( \left\Vert
Q\right\Vert ^{2}+\left\Vert P\right\Vert ^{2}+2\left\Vert P^{{\cal V}%
}\right\Vert ^{2}\right)  \nonumber \\
&&+\frac{1}{s(s-1)\ell (\ell -1)}\left( 2\breve{\delta}\left( N\right)
-\left\Vert {\cal T}^{{\cal V}}\right\Vert ^{2}+\left\Vert {\cal A}^{{\cal H}%
}\right\Vert ^{2}\right) .  \label{eq-CAS CSF 2}
\end{eqnarray}%
The equality holds in any of {\rm (\ref{eq-CAS CSF 1})} and {\rm (\ref%
{eq-CAS CSF 2})} at $p\in M_{1}$ if and only if the tensors ${\cal T}$ and $%
{\cal A}$ satisfy {\rm (\ref{eq-Cas-equality-1})}, {\rm (\ref%
{eq-Cas-equality-2})} and {\rm (\ref{eq-Cas-equality-3})}.
\end{theorem}

\begin{proof} Using (\ref{eq-GCSF}), (\ref{eq-scalV}) and (\ref%
{eq-Hscal}), we obtain%
\begin{equation}
2\tau _{{\cal V}}^{M_{1}}\left( p\right) =\frac{c}{4}\ell \left( \ell
-1\right) +\frac{3c}{4}\left\Vert Q\right\Vert ^{2},\quad 2\tau _{{\cal H}%
}^{M_{1}}\left( p\right) =\frac{c}{4}s\left( s-1\right) +\frac{3c}{4}%
\left\Vert P\right\Vert ^{2},  \label{eq-scalCSF 1}
\end{equation}%
\begin{equation}
\sum_{i=1}^{s}\sum_{j=1}^{\ell }R^{M_{1}}\left(
U_{i},V_{j},V_{j},U_{i}\right) =\frac{c}{4}s\ell +\frac{3c}{4}\left\Vert P^{%
{\cal V}}\right\Vert ^{2}.  \label{eq-NSCALCSF 1}
\end{equation}%
From (\ref{eq-scalCSF 1}) and (\ref{eq-NSCALCSF 1}), we obtain%
\begin{equation}
2\tau ^{M_{1}}\left( p\right) =\frac{c}{4}\left( \ell ^{2}+s^{2}+2\ell
s-\ell -s\right) +\frac{3c}{4}\left( \left\Vert Q\right\Vert ^{2}+\left\Vert
P\right\Vert ^{2}+2\left\Vert P^{{\cal V}}\right\Vert ^{2}\right) .
\label{eq-NSCALCSF}
\end{equation}%
Substituting (\ref{eq-NSCALCSF}) into (\ref{Gen Ine Cas 1}) and (\ref{Gen Ine Cas 2}), we get (\ref{eq-CAS CSF 1}) and (\ref{eq-CAS CSF 2}), respectively.
\end{proof}

We construct the following example to characterize the equality cases of the
above theorem.

\begin{example}
Let 
\[
M_{1}=\left\{ (x^{1},x^{2},x^{3},x^{4},x^{5},x^{6},x^{7},x^{8})\in {\Bbb R}%
^{8}:\ x^{i}>0,\ 1\leq i\leq 8\right\} , 
\]%
be equipped with the Euclidean metric 
\[
g_{1}=\sum_{i=1}^{8}\left( dx^{i}\right) ^{2}, 
\]%
and the compatible almost complex structure $J$ defined by%
\begin{eqnarray*}
J\left( \frac{\partial }{\partial x^{1}}\right) &=&\frac{\partial }{\partial
x^{2}},\ J\left( \frac{\partial }{\partial x^{2}}\right) =-\frac{\partial }{%
\partial x^{1}},\ J\left( \frac{\partial }{\partial x^{3}}\right) =\frac{%
\partial }{\partial x^{4}},\ J\left( \frac{\partial }{\partial x^{4}}\right)
=-\frac{\partial }{\partial x^{3}}, \\
J\left( \frac{\partial }{\partial x^{5}}\right) &=&\frac{\partial }{\partial
x^{6}},\ J\left( \frac{\partial }{\partial x^{6}}\right) =-\frac{\partial }{%
\partial x^{5}},\ J\left( \frac{\partial }{\partial x^{7}}\right) =\frac{%
\partial }{\partial x^{8}},\ J\left( \frac{\partial }{\partial x^{8}}\right)
=-\frac{\partial }{\partial x^{7}}.
\end{eqnarray*}%
Then $(M_{1},g_{1},J)$ becomes a K\"{a}hler manifold. Also, let 
\[
M_{2}=\left\{ (y^{1},y^{2},y^{3},y^{4})\in {\Bbb R}^{4}:\ y^{i}>0,\ 1\leq
i\leq 4\right\} , 
\]%
be endowed with the Euclidean metric 
\[
g_{2}=\sum_{j=1}^{4}\left( dy^{j}\right) ^{2}. 
\]%
Then $(M_{2},g_{2})$ becomes a Riemannian manifold. Define a map 
\[
\pi :(M_{1},g_{1},J)\longrightarrow (M_{2},g_{2}), 
\]%
by 
\[
\pi (x^{1},x^{2},x^{3},x^{4},x^{5},x^{6},x^{7},x^{8})=\left( \tau _{1},\tau
_{2},\tau _{3},\tau _{4}\right) . 
\]%
where $\tau _{1}=\sqrt{\left( x^{1}\right) ^{2}+\left( x_{2}\right) ^{2}}$, $%
\tau _{2}=\sqrt{\left( x^{3}\right) ^{2}+\left( x^{4}\right) ^{2}}$, $\tau
_{3}=\sqrt{\left( x^{5}\right) ^{2}+\left( x^{6}\right) ^{2}}$ and $\tau
_{4}=\sqrt{\left( x^{7}\right) ^{2}+\left( x^{8}\right) ^{2}}$. We obtain 
\begin{eqnarray*}
{\cal V} &=&{\rm span}\left\{ V_{1}=\frac{x^{2}}{\tau _{1}}\frac{\partial }{%
\partial x^{1}}-\frac{x^{1}}{\tau _{1}}\frac{\partial }{\partial x^{2}}%
,V_{2}=\frac{x^{4}}{\tau _{2}}\frac{\partial }{\partial x^{3}}-\frac{x^{3}}{%
\tau _{2}}\frac{\partial }{\partial x^{4}},\right. \\
&&\left. V_{3}=\frac{x^{6}}{\tau _{3}}\frac{\partial }{\partial x^{5}}-\frac{%
x^{5}}{\tau _{3}}\frac{\partial }{\partial x^{6}},\ V_{4}=\frac{x^{8}}{\tau
_{4}}\frac{\partial }{\partial x^{7}}-\frac{x^{7}}{\tau _{4}}\frac{\partial 
}{\partial x^{8}}\right\} ,
\end{eqnarray*}%
\begin{eqnarray*}
{\cal H} &=&{\rm span}\left\{ U_{1}=\frac{x^{1}}{\tau _{1}}\frac{\partial }{%
\partial x^{1}}+\frac{x^{2}}{\tau _{1}}\frac{\partial }{\partial x^{2}}%
,U_{2}=\frac{x^{3}}{\tau _{2}}\frac{\partial }{\partial x^{3}}+\frac{x^{4}}{%
\tau _{2}}\frac{\partial }{\partial x^{4}},\right. \\
&&\left. U_{3}=\frac{x^{5}}{\tau _{3}}\frac{\partial }{\partial x^{5}}+\frac{%
x^{6}}{\tau _{3}}\frac{\partial }{\partial x^{6}},U_{4}=\frac{x^{7}}{\tau
_{4}}\frac{\partial }{\partial x^{7}}+\frac{x^{8}}{\tau _{4}}\frac{\partial 
}{\partial x^{8}}\right\} ,
\end{eqnarray*}%
\[
{\rm range}~\pi _{\ast }={\rm span}\left\{ \pi _{\ast }(U_{j})=\frac{%
\partial }{\partial y^{j}}\right\} _{j=1}^{4}, 
\]%
where $\left\{ \frac{\partial }{\partial x^{i}}\right\} _{i=1}^{8}$ and $%
\left\{ \frac{\partial }{\partial y^{j}}\right\} _{j=1}^{4}$ are bases of $%
T_{p}M_{1}$ and ${\cal T}_{\pi (p)}M_{2}$ respectively. One can verify that $%
\pi $ is a Riemannian submersion with horizontal and vertical spaces of
dimensions $4$. Since all Christoffel symbols vanish, we obtain 
\[
{\cal T}_{V_{1}}^{{\cal H}}V_{1}=-\left( \frac{x^{1}}{\tau _{1}^{2}}\frac{%
\partial }{\partial x^{1}}+\frac{x^{2}}{\tau _{1}^{2}}\frac{\partial }{%
\partial x^{2}}\right) ,\quad {\cal T}_{V_{2}}^{{\cal H}}V_{2}=-\left( \frac{%
x^{3}}{\tau _{2}^{2}}\frac{\partial }{\partial x^{3}}+\frac{x^{4}}{\tau
_{2}^{2}}\frac{\partial }{\partial x^{4}}\right) , 
\]%
\[
{\cal T}_{V_{3}}^{{\cal H}}V_{3}=-\left( \frac{x^{5}}{\tau _{3}^{2}}\frac{%
\partial }{\partial x^{5}}+\frac{x^{6}}{\tau _{3}^{2}}\frac{\partial }{%
\partial x^{6}}\right) ,\quad {\cal T}_{V_{4}}^{{\cal H}}V_{4}=-\left( \frac{%
x^{7}}{\tau _{4}^{2}}\frac{\partial }{\partial x^{7}}+\frac{x^{8}}{\tau
_{4}^{2}}\frac{\partial }{\partial x^{8}}\right) , 
\]%
\[
{\cal T}_{V_{i}}^{{\cal H}}V_{j}=0~\text{ {\rm for all} $i\neq j$}. 
\]%
The only non-zero components of ${\cal T}^{\cal H}$ are
\[
\left( {\cal T}^{%
{\cal H}}\right) _{ii}^{i}=-\frac{1}{\tau _{i}}, \quad 1 \leq i \leq 4.
\]
Similarly to the previous example, we obtain ${\cal A}_{U_{i}}^{{\cal V}%
}U_{j}=0$ for all $U_{i},U_{j}\in {\cal H}$. Hence, we observe that the
inequality obtained in Theorem \ref{Theorem CSF} does not attain equality.
\end{example}
\begin{theorem}
\label{Theorem GSSF} Let $\pi :(M_{1},g_{1},J)\rightarrow \left(
M_{2},g_{2}\right) $ be a Riemannian submersion from a generalized Sasakian
space form onto a Riemannian manifold, where $\dim M_{1}=n_{1}=2m+1$ and $%
\dim M_{2}=n_{2}$ with $\dim {\cal H}_{p}=s\geq 3$, and $\dim {\cal V}%
_{p}=\ell \geq 3$. If $\xi \in {\cal V}_{p}$ or $\xi \in {\cal H}_{p}$,%
\begin{eqnarray}
\frac{\rho _{{\cal H}}^{{\cal H}}}{\ell (\ell -1)}+\frac{\rho _{{\cal V}}^{%
{\cal V}}}{s(s-1)} &\leq &\frac{1}{s(s-1)}\delta _{C}^{{\cal V}}\left( \ell
-1\right) +\frac{1}{\ell (\ell -1)}\delta _{C}^{{\cal H}}\left( s-1\right) 
\nonumber \\
&&+\frac{c_{1}}{s(s-1)\ell (\ell -1)}\left( \ell ^{2}+s^{2}+2sl-\ell
-s\right)  \nonumber \\
&&+\frac{3c_{2}}{\ell \left( \ell -1\right) s(s-1)}\left( \left\Vert
Q\right\Vert ^{2}+\left\Vert P\right\Vert ^{2}+2\left\Vert P^{{\cal V}%
}\right\Vert ^{2}\right)  \nonumber \\
&&-\frac{2c_{3}}{\ell s\left( s-1\right) (\ell -1)}\left( \ell +s-1\right) 
\nonumber \\
&&+\frac{1}{s(s-1)\ell (\ell -1)}\left( 2\breve{\delta}\left( N\right)
-\left\Vert {\cal T}^{{\cal V}}\right\Vert ^{2}+\left\Vert {\cal A}^{{\cal H}%
}\right\Vert ^{2}\right) ,  \label{eq-New GSSF 1}
\end{eqnarray}%
and%
\begin{eqnarray}
\frac{\rho _{{\cal H}}^{{\cal H}}}{\ell (\ell -1)}+\frac{\rho _{{\cal V}}^{%
{\cal V}}}{s(s-1)} &\leq &\frac{1}{s(s-1)}\hat{\delta}_{C}^{{\cal V}}\left(
\ell -1\right) +\frac{1}{\ell (\ell -1)}\hat{\delta}_{C}^{{\cal H}}\left(
s-1\right)  \nonumber \\
&&+\frac{c_{1}}{s(s-1)\ell (\ell -1)}\left( \ell ^{2}+s^{2}+2sl-\ell
-s\right)  \nonumber \\
&&+\frac{3c_{2}}{\ell \left( \ell -1\right) s(s-1)}\left( \left\Vert
Q\right\Vert ^{2}+\left\Vert P\right\Vert ^{2}+2\left\Vert P^{{\cal V}%
}\right\Vert ^{2}\right)  \nonumber \\
&&-\frac{2c_{3}}{\ell s\left( s-1\right) (\ell -1)}\left( \ell +s-1\right) 
\nonumber \\
&&+\frac{1}{s(s-1)\ell (\ell -1)}\left( 2\breve{\delta}\left( N\right)
-\left\Vert {\cal T}^{{\cal V}}\right\Vert ^{2}+\left\Vert {\cal A}^{{\cal H}%
}\right\Vert ^{2}\right) .  \label{eq-New GSSF 2}
\end{eqnarray}%
The equality holds in any of {\rm (\ref{eq-New GSSF 1})} and {\rm (\ref%
{eq-New GSSF 2})} at $p\in M_{1}$ if and only if the tensors ${\cal T}$ and $%
{\cal A}$ satisfy {\rm (\ref{eq-Cas-equality-1})}, {\rm (\ref%
{eq-Cas-equality-2})} and {\rm (\ref{eq-Cas-equality-3})}.
\end{theorem}

\begin{proof} Using (\ref{eq-scalV}), (\ref{eq-Hscal}) and (\ref%
{eq-GSSF}), we obtain 
\begin{equation}
2\tau _{{\cal V}}^{M_{1}}(p)=\left\{ 
\begin{array}{ll}
\ell \left( \ell -1\right) c_{1}+3c_{2}\left\Vert Q\right\Vert ^{2}-2\left(
\ell -1\right) c_{3} & {\rm if}~\xi \in {\cal V}_{p}; \\ 
\ell \left( \ell -1\right) c_{1}+3c_{2}\left\Vert Q\right\Vert ^{2} & {\rm if%
}~\xi \in {\rm {\cal H}}_{p},%
\end{array}%
\right.  \label{eq-Scal GSSF 1}
\end{equation}%
\begin{equation}
2\tau _{{\cal H}}^{M_{1}}(p)=\left\{ 
\begin{array}{ll}
s\left( s-1\right) c_{1}+3c_{2}\left\Vert P\right\Vert ^{2} & {\rm if}~\xi
\in {\cal V}_{p}; \\ 
s\left( s-1\right) c_{1}+3c_{2}\left\Vert P\right\Vert ^{2}-2\left(
s-1\right) c_{3} & {\rm if}~\xi \in {\rm {\cal H}}_{p},%
\end{array}%
\right.  \label{eq-Scal GSSF 2}
\end{equation}%
\begin{equation}
\sum_{i=1}^{s}\sum_{j=1}^{\ell }R^{M_{1}}\left(
U_{i},V_{j},V_{j},U_{i}\right) =\left\{ 
\begin{array}{ll}
c_{1}s\ell +3c_{2}\left\Vert P^{{\cal V}}\right\Vert ^{2}-c_{3}s & {\rm if}%
~\xi \in {\cal V}_{p}; \\ 
c_{1}s\ell +3c_{2}\left\Vert P^{{\cal V}}\right\Vert ^{2}-c_{3}\ell & {\rm if%
}~\xi \in {\rm {\cal H}}_{p}.%
\end{array}%
\right.  \label{eq-Scal GSSF 3}
\end{equation}%
From (\ref{eq-Scal GSSF 1}), (\ref{eq-Scal GSSF 2}) and (\ref{eq-Scal GSSF 3}%
), we get $2\tau ^{M_{1}}\left( p\right) $ when $\xi \in {\cal V}_{p}$ or $%
\xi \in {\cal H}_{p}$, 
\begin{eqnarray}
2\tau ^{M_{1}}\left( p\right) &=&c_{1}\left( \ell ^{2}+s^{2}+2\ell s-\ell
-s\right) +3c_{2}\left( \left\Vert Q\right\Vert ^{2}+\left\Vert P\right\Vert
^{2}+2\left\Vert P^{{\cal V}}\right\Vert ^{2}\right)  \label{eq-scal-N-GSSF}
\\
&&-2c_{3}\left( \ell +s-1\right) .  \nonumber
\end{eqnarray}%
Substituting (\ref{eq-scal-N-GSSF}) into (\ref{Gen Ine Cas 1}) and (\ref{Gen Ine
Cas 2}), we get (\ref{eq-New GSSF 1}) and (\ref{eq-New GSSF 2}), respectively.
\end{proof}

We construct the following example to characterize equality cases of above
theorem.

\begin{example}
Let $\left( M_{1}={\Bbb R\times }_{f}{\Bbb C}^{n},\
g_{1}=dt^{2}+f^{2}\sum_{i=1}^{n}\left( \left( dx^{i}\right) ^{2}+\left(
dy^{i}\right) ^{2}\right) \right) $ be a generalised Sasakian space form
(see \cite{Alegre_Blair_Carriazo_04_IJM}) with $c_{1}=\frac{f^{\prime 2}}{%
f^{2}},c_{2}=0,c_{3}=-\frac{f^{\prime 2}}{f^{2}}+\frac{f^{\prime \prime }}{f}
$ and \newline
$\left( M_{2}={\Bbb R}^{2k},\ g_{2}=f^{2}\sum_{j=1}^{2k}\left( dZ^{j}\right)
^{2}\right) $ be a Riemannian manifold, where $f$ is a positive smooth
function on ${\Bbb R}$ and ${\Bbb C}^{n}$ is a complex space form of real
dimension $2n$. Define a map $\pi :\left( M_{1},g_{1}\right) \rightarrow
\left( M_{2},g_{2}\right) $ such that%
\[
\pi \left( t,x^{1},y^{1},\ldots ,x^{n},y^{n}\right) =\left(
x^{1},y^{1},\ldots ,x^{k},y^{k}\right) . 
\]%
Then, we obtain%
\begin{eqnarray*}
{\cal V} &=&{\rm span}\left\{ V_{1}=\frac{\partial }{\partial t},V_{2}=\frac{%
1}{f}\frac{\partial }{\partial x^{k+1}},V_{3}=\frac{1}{f}\frac{\partial }{%
\partial y^{k+1}}\right. \\
&&\left. ,\ldots ,V_{2\left( n-k\right) }=\frac{1}{f}\frac{\partial }{%
\partial x^{n}},V_{2\left( n-k\right) +1}=\frac{1}{f}\frac{\partial }{%
\partial y^{n}}\right\} ,
\end{eqnarray*}%
\[
{\cal H}={\rm span}\left\{ U_{1}=\frac{1}{f}\frac{\partial }{\partial x^{1}}%
,U_{2}=\frac{1}{f}\frac{\partial }{\partial y^{1}},\ldots ,U_{2k-1}=\frac{1}{%
f}\frac{\partial }{\partial x^{k}},U_{2k}=\frac{1}{f}\frac{\partial }{%
\partial y^{k}}\right\} , 
\]%
\[
{\rm range}~\pi _{\ast }={\rm span}\left\{ \pi _{\ast }\left( U_{i}\right) =%
\frac{1}{f}\frac{\partial }{\partial Z^{i}}\right\} _{i=1}^{2k}, 
\]%
where $\left\{ \frac{\partial }{\partial t},\frac{1}{f}\frac{\partial }{%
\partial x^{1}},\frac{1}{f}\frac{\partial }{\partial y^{1}},\ldots ,\frac{1}{%
f}\frac{\partial }{\partial x^{k}},\frac{1}{f}\frac{\partial }{\partial y^{k}%
},\frac{1}{f}\frac{\partial }{\partial x^{k+1}},\frac{1}{f}\frac{\partial }{%
\partial y^{k+1}},\ldots ,\frac{1}{f}\frac{\partial }{\partial x^{n}},\frac{1%
}{f}\frac{\partial }{\partial y^{n}}\right\} $ and $\left\{ \frac{1}{f}\frac{%
\partial }{\partial Z^{i}}\right\} _{i=1}^{2k}$ are bases of $T_{p}M_{1}$
and $T_{\pi \left( p\right) }M_{2}$, respectively. We observe that $%
g_{1}\left( U_{i},U_{j}\right) =g_{2}\left( \pi _{\ast }U_{i},\pi _{\ast
}U_{j}\right) $ for all $U_{i},U_{j}\in {\cal H}$ and ${\rm rank}\pi _{\ast
}=2k$ and dimension of vertical space is $2\left( n-k\right) +1$. Hence, $%
\pi $ is a Riemannian submersion. We have all non zero christoffel symbols
for $g_{1}$ are%
\[
\Gamma _{ii}^{t}=-ff^{\prime }\ {\rm and}\ \Gamma _{ti}^{i}=\frac{f^{\prime }%
}{f}\ {\rm for\ all}\ i\in \left\{ x_{j},y_{j}\right\} _{j=1}^{n}. 
\]%
Now, we obtain%
\[
\nabla _{V_{1}}V_{i}=0,\quad 1\leq i\leq 2\left( n-k\right) +1,\quad \nabla
_{V_{i}}V_{j}=0,\quad 2\leq i\neq j\leq 2\left( n-k\right) +1, 
\]%
\begin{eqnarray*}
\nabla _{V_{i}}V_{i} &=&-\frac{f^{\prime }}{f}\frac{\partial }{\partial t}%
,\quad 2\leq i\leq 2\left( n-k\right) +1, \\
\nabla _{V_{2i}}V_{1} &=&\frac{f^{\prime }}{f^{2}}\frac{\partial }{\partial
x^{i}},\ \nabla _{V_{2i+1}}V_{1}=\frac{f^{\prime }}{f^{2}}\frac{\partial }{%
\partial y_{i}},\quad 1\leq i\leq \left( n-k\right) ,
\end{eqnarray*}%
\[
\nabla _{U_{i}}U_{j}=0,\quad 1\leq i\neq j\leq 2k\quad {\rm and}\quad \nabla
_{U_{i}}U_{i}=0,\quad 1\leq i\leq 2k, 
\]%
This implies that $h\nabla _{V_{i}}V_{j}={\cal T}_{V_{i}}^{{\cal H}}V_{j}=0$
for all $1\leq i,j\leq 2\left( n-k\right) +1$. Hence, we have%
\begin{eqnarray*}
\left( {\cal T}^{{\cal H}}\right) _{ij}^{\alpha } &=&0,\quad 1\leq i,j\leq
2\left( n-k\right) +1,\quad \alpha \in \left\{ 1,\ldots ,2k\right\} , \\
\left( {\cal A}^{{\cal V}}\right) _{ij}^{\alpha } &=&0,\quad i,j\in \left\{
1,\ldots ,2k\right\} ,\quad 1\leq \alpha \leq 2\left( n-k\right) +1.
\end{eqnarray*}%
We observe that the inequality obtained in Theorem \ref{Theorem GSSF}
attains equality when $2\left( n-k\right) +1\geq 3$ and $2k\geq 3$.
\end{example}

By assuming the suitable values of $c_{1}$, $c_{2}$ and $c_{3}$ from table %
\ref{Table 1}, We have the following corollary:

\begin{corollary}
Let $\pi :(M_{1},g_{1})\rightarrow \left( M_{2},g_{2}\right) $ be a
Riemannian submersion, where $\dim M_{1}=n_{1}=2m+1$ and $\dim M_{2}=n_{2}$
with $\dim {\cal H}_{p}=s\geq 3$ and $\dim {\cal V}_{p}=\ell \geq 3$,

\begin{itemize}
\item[{\bf (1)}] If $M_{1}$ is sasakian space form. If $\xi \in {\cal V}_{p}$
or $\xi \in {\cal H}_{p}$,%
\begin{eqnarray}
\frac{\rho _{{\cal H}}^{{\cal H}}}{\ell (\ell -1)}+\frac{\rho _{{\cal V}}^{%
{\cal V}}}{s(s-1)} &\leq &\frac{1}{s(s-1)}\delta _{C}^{{\cal V}}\left( \ell
-1\right) +\frac{1}{\ell (\ell -1)}\delta _{C}^{{\cal H}}\left( s-1\right) 
\nonumber \\
&&+\frac{c+3}{4s(s-1)\ell (\ell -1)}\left( \ell ^{2}+s^{2}+2sl-\ell -s\right)
\nonumber \\
&&+\frac{3\left( c-1\right) }{4\ell \left( \ell -1\right) s(s-1)}\left(
\left\Vert Q\right\Vert ^{2}+\left\Vert P\right\Vert ^{2}+2\left\Vert P^{%
{\cal V}}\right\Vert ^{2}\right)  \nonumber \\
&&-\frac{2\left( c-1\right) }{4\ell s\left( s-1\right) (\ell -1)}\left( \ell
+s-1\right)  \nonumber \\
&&+\frac{1}{s(s-1)\ell (\ell -1)}\left( 2\breve{\delta}\left( N\right)
-\left\Vert {\cal T}^{{\cal V}}\right\Vert ^{2}+\left\Vert {\cal A}^{{\cal H}%
}\right\Vert ^{2}\right) ,  \label{eq-SAS-(1)}
\end{eqnarray}%
and%
\begin{eqnarray}
\frac{\rho _{{\cal H}}^{{\cal H}}}{\ell (\ell -1)}+\frac{\rho _{{\cal V}}^{%
{\cal V}}}{s(s-1)} &\leq &\frac{1}{s(s-1)}\hat{\delta}_{C}^{{\cal V}}\left(
\ell -1\right) +\frac{1}{\ell (\ell -1)}\hat{\delta}_{C}^{{\cal H}}\left(
s-1\right)  \nonumber \\
&&+\frac{c+3}{4s(s-1)\ell (\ell -1)}\left( \ell ^{2}+s^{2}+2sl-\ell -s\right)
\nonumber \\
&&+\frac{3\left( c-1\right) }{4\ell \left( \ell -1\right) s(s-1)}\left(
\left\Vert Q\right\Vert ^{2}+\left\Vert P\right\Vert ^{2}+2\left\Vert P^{%
{\cal V}}\right\Vert ^{2}\right)  \nonumber \\
&&-\frac{2\left( c-1\right) }{4\ell s\left( s-1\right) (\ell -1)}\left( \ell
+s-1\right)  \nonumber \\
&&+\frac{1}{s(s-1)\ell (\ell -1)}\left( 2\breve{\delta}\left( N\right)
-\left\Vert {\cal T}^{{\cal V}}\right\Vert ^{2}+\left\Vert {\cal A}^{{\cal H}%
}\right\Vert ^{2}\right) .  \label{eq-SAS-(2)}
\end{eqnarray}

\item[{\bf (2)}] If $M_{1}$ is Kenmotsu space form. If $\xi \in {\cal V}_{p}$
or $\xi \in {\cal H}_{p}$,%
\begin{eqnarray}
\frac{\rho _{{\cal H}}^{{\cal H}}}{\ell (\ell -1)}+\frac{\rho _{{\cal V}}^{%
{\cal V}}}{s(s-1)} &\leq &\frac{1}{s(s-1)}\delta _{C}^{{\cal V}}\left( \ell
-1\right) +\frac{1}{\ell (\ell -1)}\delta _{C}^{{\cal H}}\left( s-1\right) 
\nonumber \\
&&+\frac{c-3}{4s(s-1)\ell (\ell -1)}\left( \ell ^{2}+s^{2}+2sl-\ell -s\right)
\nonumber \\
&&+\frac{3\left( c+1\right) }{4\ell \left( \ell -1\right) s(s-1)}\left(
\left\Vert Q\right\Vert ^{2}+\left\Vert P\right\Vert ^{2}+2\left\Vert P^{%
{\cal V}}\right\Vert ^{2}\right)  \nonumber \\
&&-\frac{2\left( c+1\right) }{4\ell s\left( s-1\right) (\ell -1)}\left( \ell
+s-1\right)  \nonumber \\
&&+\frac{1}{s(s-1)\ell (\ell -1)}\left( 2\breve{\delta}\left( N\right)
-\left\Vert {\cal T}^{{\cal V}}\right\Vert ^{2}+\left\Vert {\cal A}^{{\cal H}%
}\right\Vert ^{2}\right) ,  \label{eq-Ken-(1)}
\end{eqnarray}%
and%
\begin{eqnarray}
\frac{\rho _{{\cal H}}^{{\cal H}}}{\ell (\ell -1)}+\frac{\rho _{{\cal V}}^{%
{\cal V}}}{s(s-1)} &\leq &\frac{1}{s(s-1)}\hat{\delta}_{C}^{{\cal V}}\left(
\ell -1\right) +\frac{1}{\ell (\ell -1)}\hat{\delta}_{C}^{{\cal H}}\left(
s-1\right)  \nonumber \\
&&+\frac{c-3}{4s(s-1)\ell (\ell -1)}\left( \ell ^{2}+s^{2}+2sl-\ell -s\right)
\nonumber \\
&&+\frac{3\left( c+1\right) }{4\ell \left( \ell -1\right) s(s-1)}\left(
\left\Vert Q\right\Vert ^{2}+\left\Vert P\right\Vert ^{2}+2\left\Vert P^{%
{\cal V}}\right\Vert ^{2}\right)  \nonumber \\
&&-\frac{2\left( c+1\right) }{4\ell s\left( s-1\right) (\ell -1)}\left( \ell
+s-1\right)  \nonumber \\
&&+\frac{1}{s(s-1)\ell (\ell -1)}\left( 2\breve{\delta}\left( N\right)
-\left\Vert {\cal T}^{{\cal V}}\right\Vert ^{2}+\left\Vert {\cal A}^{{\cal H}%
}\right\Vert ^{2}\right) .  \label{eq-Ken-(2)}
\end{eqnarray}

\item[{\bf (3)}] If $M_{1}$ is cosymplectic space form. If $\xi \in {\cal V}%
_{p}$ or $\xi \in {\cal H}_{p}$,%
\begin{eqnarray}
\frac{\rho _{{\cal H}}^{{\cal H}}}{\ell (\ell -1)}+\frac{\rho _{{\cal V}}^{%
{\cal V}}}{s(s-1)} &\leq &\frac{1}{s(s-1)}\delta _{C}^{{\cal V}}\left( \ell
-1\right) +\frac{1}{\ell (\ell -1)}\delta _{C}^{{\cal H}}\left( s-1\right) 
\nonumber \\
&&+\frac{c}{4s(s-1)\ell (\ell -1)}\left( \ell ^{2}+s^{2}+2sl-\ell -s\right) 
\nonumber \\
&&+\frac{3c}{4\ell \left( \ell -1\right) s(s-1)}\left( \left\Vert
Q\right\Vert ^{2}+\left\Vert P\right\Vert ^{2}+2\left\Vert P^{{\cal V}%
}\right\Vert ^{2}\right)  \nonumber \\
&&-\frac{2c}{4\ell s\left( s-1\right) (\ell -1)}\left( \ell +s-1\right) 
\nonumber \\
&&+\frac{1}{s(s-1)\ell (\ell -1)}\left( 2\breve{\delta}\left( N\right)
-\left\Vert {\cal T}^{{\cal V}}\right\Vert ^{2}+\left\Vert {\cal A}^{{\cal H}%
}\right\Vert ^{2}\right) ,  \label{eq-cos-(1)}
\end{eqnarray}%
and%
\begin{eqnarray}
\frac{\rho _{{\cal H}}^{{\cal H}}}{\ell (\ell -1)}+\frac{\rho _{{\cal V}}^{%
{\cal V}}}{s(s-1)} &\leq &\frac{1}{s(s-1)}\hat{\delta}_{C}^{{\cal V}}\left(
\ell -1\right) +\frac{1}{\ell (\ell -1)}\hat{\delta}_{C}^{{\cal H}}\left(
s-1\right)  \nonumber \\
&&+\frac{c}{4s(s-1)\ell (\ell -1)}\left( \ell ^{2}+s^{2}+2sl-\ell -s\right) 
\nonumber \\
&&+\frac{3c}{4\ell \left( \ell -1\right) s(s-1)}\left( \left\Vert
Q\right\Vert ^{2}+\left\Vert P\right\Vert ^{2}+2\left\Vert P^{{\cal V}%
}\right\Vert ^{2}\right)  \nonumber \\
&&-\frac{2c}{4\ell s\left( s-1\right) (\ell -1)}\left( \ell +s-1\right) 
\nonumber \\
&&+\frac{1}{s(s-1)\ell (\ell -1)}\left( 2\breve{\delta}\left( N\right)
-\left\Vert {\cal T}^{{\cal V}}\right\Vert ^{2}+\left\Vert {\cal A}^{{\cal H}%
}\right\Vert ^{2}\right) .  \label{eq-cos-(2)}
\end{eqnarray}

\item[{\bf (4)}] If $M_{1}$ is almost $C\left( \alpha \right) $ space form.
If $\xi \in {\cal V}_{p}$ or $\xi \in {\cal H}_{p}$,%
\begin{eqnarray}
\frac{\rho _{{\cal H}}^{{\cal H}}}{\ell (\ell -1)}+\frac{\rho _{{\cal V}}^{%
{\cal V}}}{s(s-1)} &\leq &\frac{1}{s(s-1)}\delta _{C}^{{\cal V}}\left( \ell
-1\right) +\frac{1}{\ell (\ell -1)}\delta _{C}^{{\cal H}}\left( s-1\right) 
\nonumber \\
&&+\frac{c+3\alpha }{4s(s-1)\ell (\ell -1)}\left( \ell
^{2}+s^{2}+2sl-\ell -s\right)  \nonumber \\
&&+\frac{3\left( c-\alpha \right) }{4\ell \left( \ell -1\right) s(s-1)}%
\left( \left\Vert Q\right\Vert ^{2}+\left\Vert P\right\Vert ^{2}+2\left\Vert
P^{{\cal V}}\right\Vert ^{2}\right)  \nonumber \\
&&-\frac{2\left( c-\alpha \right) }{4\ell s\left( s-1\right) (\ell -1)}%
\left( \ell +s-1\right)  \nonumber \\
&&+\frac{1}{s(s-1)\ell (\ell -1)}\left( 2\breve{\delta}\left( N\right)
-\left\Vert {\cal T}^{{\cal V}}\right\Vert ^{2}+\left\Vert {\cal A}^{{\cal H}%
}\right\Vert ^{2}\right) ,  \label{eq-calpha-(1)}
\end{eqnarray}%
and%
\begin{eqnarray}
\frac{\rho _{{\cal H}}^{{\cal H}}}{\ell (\ell -1)}+\frac{\rho _{{\cal V}}^{%
{\cal V}}}{s(s-1)} &\leq &\frac{1}{s(s-1)}\hat{\delta}_{C}^{{\cal V}}\left(
\ell -1\right) +\frac{1}{\ell (\ell -1)}\hat{\delta}_{C}^{{\cal H}}\left(
s-1\right)  \nonumber \\
&&+\frac{c+3\alpha }{4s(s-1)\ell (\ell -1)}\left( \ell
^{2}+s^{2}+2sl-\ell -s\right)  \nonumber \\
&&+\frac{3\left( c-\alpha \right) }{4\ell \left( \ell -1\right) s(s-1)}%
\left( \left\Vert Q\right\Vert ^{2}+\left\Vert P\right\Vert ^{2}+2\left\Vert
P^{{\cal V}}\right\Vert ^{2}\right)  \nonumber \\
&&-\frac{2\left( c-\alpha \right) }{4\ell s\left( s-1\right) (\ell -1)}%
\left( \ell +s-1\right)  \nonumber \\
&&+\frac{1}{s(s-1)\ell (\ell -1)}\left( 2\breve{\delta}\left( N\right)
-\left\Vert {\cal T}^{{\cal V}}\right\Vert ^{2}+\left\Vert {\cal A}^{{\cal H}%
}\right\Vert ^{2}\right) .  \label{eq-calpha-(2)}
\end{eqnarray}
\end{itemize}
\end{corollary}
\noindent{\bf Acknowledgement.} The first author gratefully acknowledges the Human Resource Development Group (HRDG), Council of Scientific and Industrial Research (CSIR), New Delhi, India, for awarding the CSIR Junior Research Fellowship (File No. 09/0013(16054)/2022-EMR-I). The second author gratefully acknowledges Banaras Hindu University for providing the ``Incentive Grant'' under the IoE Scheme (BHU), Ref. No. R/Dev/D/IoE/Incentive (Phase-IV)/2024-25/82489.

\end{document}